\documentclass[hidelinks]{article}
\usepackage[utf8]{inputenc}
\usepackage{titlesec}
\usepackage[breaklinks=true]{hyperref}
\usepackage[includehead,a4paper,headheight=13.6pt]{geometry} 
\usepackage[indent=15pt]{parskip} 
\usepackage{amssymb,amsmath,amsthm}
\usepackage[nameinlink]{cleveref}
\usepackage{indentfirst}
\usepackage{xcolor}

\newtheorem{thm}{Theorem}
\newtheorem{lem}[thm]{Lemma}
\newtheorem{cor}[thm]{Corollary}
\newtheorem{defn}[thm]{Definition}

\titleformat*{\section}{\large\bfseries}
\titleformat*{\subsection}{\normalsize\bfseries}
\title{\vspace{-2em}Sum of Squares Decompositions in H\"older Spaces}
\author{Sullivan Francis MacDonald\footnote{University of Toronto, ON, Canada. sullivan@math.utoronto.ca. The author was supported by the Britton Estate Scholarship at McMaster University and the Department of Mathematics at the University of Toronto.}}
\date{September 2023}

\begin{document}
\maketitle

\begin{abstract}
We investigate the number of half-regular squares required to decompose a non-negative \(C^{k,\alpha}(\mathbb{R}^n)\) function into a sum of squares. Each non-negative \(C^{3,1}(\mathbb{R}^n)\) function is known to be a finite SOS in \(C^{1,1}(\mathbb{R}^n)\), and similar regularity-preserving SOS decompositions have been studied by various authors. Our work refines existing techniques to unify and build upon several known decomposition results, and moreover we provide upper and lower estimates on the number of squares required for SOS decompositions in \(C^{k,\alpha}(\mathbb{R}^n)\).
\end{abstract}

\section{Introduction}

In this paper we are concerned with the problem of decomposing a given non-negative \(C^{k,\alpha}(\mathbb{R}^n)\) function into a sum of squares (SOS) of `half-regular' functions belonging to the space 
\begin{equation}\label{halfreg}
    C^\frac{k+\alpha}{2}(\mathbb{R}^n)=\begin{cases}\hfil C^{\frac{k}{2},\frac{\alpha}{2}}(\mathbb{R}^n) & k\textrm{ even}, \\ C^{\frac{k-1}{2},\frac{1+\alpha}{2}}(\mathbb{R}^n) & k\textrm{ odd}.
    \end{cases}
\end{equation}
Our main goals are to determine when such a decomposition is possible, and to estimate the number of functions which must be used in a SOS decomposition. That is, given \(f\in C^{k,\alpha}(\mathbb{R}^n)\) which is a SOS of functions in \eqref{halfreg}, what is the smallest integer \(m=m(n,k,\alpha)\) for which there exist functions \(g_1,\dots, g_m\) in \eqref{halfreg} such that \(f=g_1^2+\cdots+g_m^2\)?

Such regularity preserving SOS decompositions have been studied by several authors. Fefferman \& Phong \cite{Fefferman-Phong} showed that \(m(n,3,1)<\infty\) for any \(n\in\mathbb{N}\), meaning that any non-negative \(C^{3,1}(\mathbb{R}^n)\) function is a finite SOS of functions in \(C^{1,1}(\mathbb{R}^n)\). This result has found several applications in PDEs and functional analysis; see e.g. \cite{Fefferman-Phong,Guan,Tataru}. Further, Bony \cite{BonyFr} shows that \(m(1,k,\alpha)=2\) for \(\alpha\in(0,1]\) and \(k\geq 2\), resolving the problem in one dimension. Techniques similar to those used by Glaeser \cite{Glaeser} similarly show that \(m(1,1,1)=1\). In contrast, it seems that little is known about \(m(n,k,\alpha)\) when \(n\) and \(k\) are both large.

The constant \(m(n,k,\alpha)\) can be viewed in loose analogy to the Pythagoras number of a field \(F\), which is the smallest integer \(p\) for which each SOS in \(F\) is a SOS of at most \(p\) elements. Pythagoras numbers have been studied extensively in algebraic settings; famously, Lagrange's four-square theorem shows that the Pythagoras number of \(\mathbb{Q}\) is four, and these invariants also arise in connection to Hilbert's 17\(^\textrm{th}\) problem concerning sums of squares of rational functions. For a complete discussion, the interested reader is referred to \cite{Pfister}. Much of our terminology is borrowed from the literature on algebraic sums of squares.

Our motivation for studying the problem outlined above is twofold. First, there is active interest in studying regularity preserving SOS decompositions for applications elsewhere in analysis; they have been used to show that certain degenerate elliptic operators are hypoelliptic  \cite{SOS_III}, they have been employed to study regularity of solutions to the Monge-Amp\`ere equation \cite{Guan}, and famously they were used by Fefferman \& Phong  \cite{Fefferman-Phong} to prove sharpened forms of G{\aa}rding's inequality. Second, we wish to explore the connections between algebraic sums of squares and regularity preserving SOS decompositions further, and to refine standard decomposition techniques to minimize the number of squares needed for a decomposition.

It is not always possible to decompose a given \(C^{k,\alpha}(\mathbb{R}^n)\) function into a finite SOS of functions in \eqref{halfreg} when \(k\geq 4\); see Theorem \ref{thm:existencecounter}. Rather, for our purposes it is necessary to study the family of \(C^{k,\alpha}(\mathbb{R}^n)\) functions which can be decomposed, which we call \(\Sigma(n,k,\alpha)\). That is,
\[
    \Sigma(n,k,\alpha)=\bigg\{f\in C^{k,\alpha}(\mathbb{R}^n):\;f=\sum_{j=1}^s g_j^2\;\textrm{ for }s<\infty\textrm{ and }g_j\in C^\frac{k+\alpha}{2}(\mathbb{R}^n)\bigg\}.
\]
The result of Fefferman \& Phong in \cite{Fefferman-Phong} shows that \(\Sigma(n,3,1)=C^{3,1}(\mathbb{R}^n)\), and later we prove that \(\Sigma(n,k,\alpha)=C^{k,\alpha}(\mathbb{R}^n)\) whenever \(k+\alpha\leq 4\). In general it is only true that \(\Sigma(n,k,\alpha)\subseteq C^{k,\alpha}(\mathbb{R}^n)\), as the examples in \cite{BonyEn} show. Following \cite{Pfister}, we define the length \(\ell(f)\) of a function \(f\in\Sigma(n,k,\alpha)\) to be the smallest number of squares needed to represent it:
\[
    \ell(f)=\min\bigg\{s\in\mathbb{N}:\exists\; g_1,\dots,g_s\in C^{\frac{k+\alpha}{2}}(\mathbb{R}^n)\textrm{ such that }f=\sum_{j=1}^s g_j^2\bigg\}.
\]
Explicitly, the constants we study are defined by \(m(n,k,\alpha)=\displaystyle\sup\{\ell(f)\;:\;f\in \Sigma(n,k,\alpha)\}\).

There is no known characterization of \(\Sigma(n,k,\alpha)\) when \(k\geq 4\), but some tangible subsets have been identified; Korobenko \& Sawyer \cite{SOS_I} adapted the arguments of Bony et. al. \cite{BonyFr,BonyEn} and Fefferman \& Phong \cite{Fefferman-Phong} to show that functions in \(C^{4,\alpha}(\mathbb{R}^2)\) which satisfy certain differential inequalities can be decomposed into finite sums of half-regular squares; we generalize their result in Theorem \ref{main 3}, and gain additional information about \(\Sigma(n,k,\alpha)\) when \(k\geq 4\) via Theorem \ref{main 2}.

The main results of this paper are summarized in the following theorems. Theorem \ref{main 1} treats the case \(k\leq 3\), in which every non-negative \(C^{k,\alpha}(\mathbb{R}^n)\) function can be decomposed into a finite sum of half-regular squares. Theorem \ref{main 2} provides lower bounds on \(m(n,k,\alpha)\) using preexisting polynomial theory together with compactness arguments. Theorem \ref{main 3} partially extends Theorem \ref{main 1} to the case \(k\geq 4\) in two dimensions. Finally, Theorem \ref{main 4} shows that \(m(n,1,\alpha)=1\) always.

\begin{thm}\label{main 1}
If \(k\leq 3\), \(\alpha\in(0,1]\) and \(n\geq 1\), then each non-negative \(C^{k,\alpha}(\mathbb{R}^n)\) function is a finite sum of squares of functions in \(C^\frac{k+\alpha}{2}(\mathbb{R}^n)\). The number of squares required is at most
\[
    m(n,k,\alpha)\leq 2^{n^2+n-1}.
\]
\end{thm}

\begin{thm}\label{main 2}
Let \(k\geq 2\), \(\alpha\in (0,1]\), and \(n\geq 1\). There exist functions in \(C^{k,\alpha}(\mathbb{R}^n)\cap \Sigma(n,k,\alpha)\) which are sums of squares of no fewer than \(m(n,k,\alpha)\) functions in the space \(C^{\frac{k+\alpha}{2}}(\mathbb{R}^n)\), where
\[
    m(n,k,\alpha)\geq 2^{n-1}\left(\frac{k}{k+n}\right)^{\frac{n}{2}}.
\]
\end{thm}

Sharper bounds for \(m(n,k,\alpha)\) arise in our proofs of the theorems above, but these are difficult to express in closed form. For instance, we show that if \(k\leq3\) then \(m(2,k,\alpha)\leq 27\). Likewise, restricting to the case \(k=2\), the estimates of Section 7 show that for any \(n\geq 1\) we have
\[
    \frac{n+1}{2}\leq m(n,2,\alpha)\leq  2^{n^2+n-1}.
\]

When \(k\geq 4\) our methods do not give an upper bound on \(m(n,k,\alpha)\), but using essentially the same techniques that we use to prove Theorem \ref{main 1} it is possible to bound the number of squares required to decompose certain \(C^{k,\alpha}(\mathbb{R}^2)\) functions. We require that these functions satisfy generalized versions of the inequalities identified by Korobenko \& Sawyer \cite[Thm. 4.5]{SOS_I}.

\begin{thm}\label{main 3}
For \(k\geq 4\) let \(f\in C^{k,\alpha}(\mathbb{R}^2)\) be non-negative, and assume that for \(x\in\mathbb{R}^n\) \(f\) satisfies
\begin{equation}\label{diffeqs}
    |\nabla^\ell f(x)|\leq Cf(x)^\frac{k-\ell+\alpha}{k+\alpha}    
\end{equation}
for every even \(\ell\geq 4\). Then \(f\) is a sum of squares of at most 27 functions in \(C^{\frac{k+\alpha}{2}}(\mathbb{R}^2)\).
\end{thm}

It is not clear whether \(m(n,k,\alpha)\) actually depends on \(\alpha\) in general; the bounds above are all independent of \(\alpha\), but this does not preclude some dependency which our techniques overlook. In the case \(k=1\) however, matters are considerably simpler.

\begin{thm}\label{main 4}
If \(\alpha\in(0,1]\), \(n\geq 1\) and \(f\in C^{1,\alpha}(\mathbb{R}^n)\) is non-negative, then \(\sqrt{f}\in C^{\frac{1+\alpha}{2}}(\mathbb{R}^n)\). That is, for any \(\alpha\in (0,1]\) and \(n\in\mathbb{N}\) we have \(m(n,1,\alpha)=1\).
\end{thm}

The remainder of the paper is organized as follows. In Sections 2 and 3 we state preliminary results, and in Section 4 we study localized decompositions of \(C^{k,\alpha}(\mathbb{R}^n)\) functions by generalizing the technique of Fefferman \& Phong \cite{Fefferman-Phong}. Using these preliminaries, the remaining Sections 5 through 9 are devoted to proving our main theorems.

Many of the techniques in this paper are refined forms of those introduced in \cite{BonyFr,BonyEn,Fefferman-Phong,SOS_I,Tataru}, and the other cited works on SOS decompositions. Our results enjoy considerable generality thanks to the hard work of these authors, who pioneered the results that we can now build upon.

\section{Preliminaries}

\subsection*{H\"older Spaces}

We define H\"older spaces in standard fashion. Given \(\alpha\in(0,1]\) and a domain \(\Omega\subseteq\mathbb{R}^n\), the H\"older space \(C^\alpha(\Omega)\) is the real vector space of functions \(f:\Omega\rightarrow\mathbb{R}\) for which
\[
    [f]_{\alpha,\Omega}=\sup_{\substack{x,y\in\Omega\\x\neq y}}\frac{|f(x)-f(y)|}{|x-y|^\alpha}<\infty.
\]
To describe higher-order H\"older spaces concisely we use multi-index notation. A multi-index is an \(n\)-tuple \(\beta\in \mathbb{N}_0^n\), and given \(\beta=(\beta_1,\dots,\beta_n)\) the order of \(\beta\) is given by \(|\beta|=\beta_1+\cdots\beta_n\). The \(\beta\)-derivative of a function \(f\) is formally defined by
\[
    \partial^\beta f=\frac{\partial^{|\beta|}f}{\partial^{\beta_1}x_1\cdots\partial^{\beta_n}x_n}.
\]
Given \(k\in\mathbb{N}\), the H\"older space \(C^{k,\alpha}(\Omega)\) is defined to be the collection of functions \(f:\Omega\rightarrow\mathbb{R}\) for which \([\partial^\beta f]_{\alpha,\Omega}<\infty\) whenever \(|\beta|=k\).

In our later arguments, it will be helpful to employ a pointwise variant of the H\"older semi-norm. Introduced in \cite{BonyFr} and used in \cite{SOS_I}, this object is defined for a function \(f:\Omega\rightarrow\mathbb{R}\) and \(x\in\Omega\) by setting
\[
    [f]_\alpha(x)=\limsup_{z,y\rightarrow x}\frac{|f(z)-f(y)|}{|z-y|^\alpha}.
\]
It is straightforward to show that \([f]_\alpha(x)\leq [f]_{\alpha,\Omega}\) and indeed, that \([f]_{\alpha,\Omega}=\displaystyle\sup_{x\in\Omega}[f]_\alpha(x)\).

We also recall a useful property regarding compact embeddings of H\"older spaces. Denoting by \(C_b^{k,\alpha}(\Omega)\) the subspace of \(C^{k,\alpha}(\Omega)\) comprised functions for which
\[
    \|f\|_{C^{k,\alpha}_b(\Omega)}=\sum_{|\beta|\leq k}\sup_\Omega|\partial^\beta f|+\sum_{|\beta|=k}[\partial^\beta f]_{\alpha,\Omega}<\infty,
\]
it is well-known that \(C^{k,\alpha}_b(\Omega)\) is compactly embedded in lower-order H\"older spaces when one restricts to precompact domains. The following result makes this notion precise; it is a variant of classical results like \cite[Lem. 6.33]{GilbargTrudinger} and \cite[Thm. 24.14]{Driver}.

\begin{lem}\label{lem:cptmb}
Let \(0\leq\alpha,\beta\leq 1\), and let \(j\) and \(k\) be non-negative integers. If \(\Omega\) is precompact and \(k+\alpha<j+\beta\), then \(C_b^{j,\beta}(\overline{\Omega})\) is compactly embedded in \(C_b^{k,\alpha}(\overline{\Omega})\).
\end{lem}

In other words, any bounded sequence in \(C_b^{k,\alpha}(\overline{\Omega})\) has a convergent subsequence in \(C_b^{j,\beta}(\overline{\Omega})\). This will be useful in Section \ref{sec:poly}.

\subsection*{Multi-Index Calculus}

To perform the necessary estimates for our main regularity results, we introduce additional terminology and notation for working with multi-index derivatives. Given \(k\)-times differentiable functions \(f\) and \(g\) and a multi-index \(\beta\in\mathbb{N}_0^n\) for which \(|\beta|=k\), the generalized product rule reads
\[
    \partial^\beta (fg)=\sum_{\gamma\leq\beta}\binom{\beta}{\gamma}(\partial^{\gamma }f)(\partial^{\beta-\gamma}g).
\]
The subscript \(\gamma\leq\beta\) indicates that we sum over all multi-indices \(\gamma\) of length \(n\) which are smaller than \(\beta\) in the partial order defined by setting \((\gamma_1,\dots,\gamma_n)\leq(\beta_1,\dots,\beta_n)\) if and only if \(\gamma_j\leq \beta_j\) for each \(j=1,\dots,n\). The generalized binomial coefficient is given by
\[
    \binom{\beta}{\gamma}=\frac{\beta!}{\gamma!(\beta-\gamma)!},
\]
where \(\beta!=\beta_1!\cdots\beta_n!\). As an easy consequence of the formula above and the definition of the pointwise semi-norm, we have a variant of the sub-product rule employed in \cite{SOS_I}:
\[
    [\partial^\beta(fg)]_{\alpha}(x)\leq\sum_{\gamma\leq\beta}\binom{\beta}{\gamma}([\partial^{\beta-\gamma}f]_{\alpha}(x)|\partial^\gamma g(x)|+[\partial^{\beta-\gamma}g]_{\alpha}(x)|\partial^\gamma f(x)|).
\]

Along with the product and sub-product rules, we require a version of the chain rule adapted to higher-order derivatives. Many variants appear in the literature, see e.g. \cite[(3.4)]{SOS_I}, and we use one adapted to multi-indices. To state it concisely we use multi-index partitions. A multi-set \(\Gamma\) of multi-indices is called a partition of \(\beta\) if \(|\gamma|>0\) for each \(\gamma\in\Gamma\), and 
\[
    \sum_{\gamma\in\Gamma}\gamma=\beta.
\]
Several multi-indices in the sum above can be identical. Given a multi-index \(\beta\), we denote the set of its distinct partitions by \(P(\beta)\).

\begin{lem}[Chain Rule]\label{lem:genchain}
Let \(f:\mathbb{R}^{n+1}\rightarrow\mathbb{R}\) and \(g:\mathbb{R}^{n}\rightarrow\mathbb{R}\) both be \(k\) times differentiable, and define \(h:\mathbb{R}^n\rightarrow\mathbb{R}\) by \(h(x)=f(x,g(x))\). For any multi-index \(\beta\) of order \(|\beta|\leq k\) and length \(n\), there exist constants \(C_{\beta,\Gamma}\) for which
\begin{equation}\label{eq:holycow}
    \partial^\beta h =\sum_{\eta\leq\beta}\sum_{\Gamma\in P(\eta)}C_{\beta,\Gamma}(\partial^{\beta-\eta}\partial^{|\Gamma|}_{n+1}f)\prod_{\gamma\in\Gamma}\partial^\gamma g,
\end{equation}
where \(\partial^\beta h\) is evaluated at \(x\) and the derivatives of \(f\) are evaluated at \((x,g(x))\). 
\end{lem}

The constants \(C_{\beta,\Gamma}\) can be found explicitly by counting multi-index partitions. Alternate but equivalent forms of the expression above for \(\partial^\beta h\) can be found in \cite{hardy,SOS_I}. For our purposes the values of constants like \(C_{\beta,\Gamma}\) are unimportant, and henceforth we follow standard convention by letting \(C\) denote a positive constant whose value may change from line to line, but which depends on fixed parameters such as \(n\), \(k\) and \(\alpha\). When the values of constants are significant, they are emphasized.

A useful consequence of Lemma \ref{lem:genchain} is a recursive formula for the derivatives of implicitly defined functions, which we include here with the standard implicit function theorem from \cite{IFTref}.

\begin{thm}[Implicit Function Theorem]\label{thm:IFT}
Let \(G\in C^k(\mathbb{R}^{n+1})\) and let \((x_0,z_0)\in\mathbb{R}^n\times\mathbb{R}\) be a point at which \(G(x_0,z_0)=0\) and \(\frac{\partial G}{\partial x_n}(x_0,z_0)>0\). There exists a unique \(g\in C^k(U)\), for some neighbourhood \(U\subset\mathbb{R}^{n}\) of \(x_0\), such that \(G(x,g(x))=0\) for every \(x\in U\). Further, the derivatives of \(g\) are given recursively for constants \(C_{\beta,\Gamma}\) by
\[
    \partial^\beta g=-\frac{1}{\partial_{n}G}\sum_{0\leq\eta\leq\beta}\sum_{ \substack{\Gamma\in P(\eta)\setminus\{\beta\}}}C_{\beta,\Gamma}(\partial^{\beta-\eta}\partial^{|\Gamma|}_nG)\prod_{\gamma\in\Gamma}\partial^\gamma g,
\]
where \(\partial^\beta g\) is evaluated at \(x\) and the derivatives of \(G\) are all evaluated at \((x,g(x))\).
\end{thm}

In the ensuing sections we will also make H\"older estimates for the derivatives of functions, and this can be done with the help of the following inequality.

\begin{lem}\label{lem:Taylorest}
Let \(f\in C^{k,\alpha}(\mathbb{R}^n)\). If \(|\beta|<k\) then for any \(x,y\in\mathbb{R}^n\) the following estimate holds:
\[
    |\partial^\beta f(x)-\partial^\beta f(y)|\leq \sum_{1\leq  |\gamma|\leq k-|\beta|}\frac{1}{\gamma!}|x-y|^{|\gamma|}|\partial^{\beta+\gamma} f(x)|+C|x-y|^{k-|\beta|+\alpha}.    
\]
\end{lem}

\begin{proof}
Let \(\beta\) be any multi-index of length \(n\) and order \(|\beta|<k\). We begin by using Taylor's theorem to form a Taylor expansion for \(\partial^\beta f\) evaluated at \(y\) and centred at \(x\),
\[
    \partial^\beta f(y)=\sum_{0\leq |\gamma|<k-|\beta|}\frac{1}{\gamma!}\partial^{\beta+\gamma} f(x)(y-x)^\gamma+\sum_{|\gamma|=k-|\beta|}\frac{|\gamma|}{\gamma!}(y-x)^\gamma\int_0^1(1-t)^{|\gamma|-1}\partial^{\beta+\gamma}f(x+t(y-x))dt.
\]
To refine this, we observe that the integral on the right-hand side can be rewritten in the form
\[
    \frac{1}{|\gamma|}\partial^{\beta+\gamma}f(x)+\int_0^1(1-t)^{|\gamma|-1}(\partial^{\beta+\gamma}f(x+t(y-x))-\partial^{\beta+\gamma}f(x))dt.
\]
Consequently, if \(|\beta|+|\gamma|=k\) we may use that \(f\in C^{k,\alpha}(\Omega)\) by assumption to make the estimate 
\[
    |\gamma|\int_0^1(1-t)^{|\gamma|-1}\partial^{\beta+\gamma}f(x+t(y-x))dt\leq |\partial^{\beta+\gamma}f(x)|+[\partial^{\beta+\gamma}f]_{\alpha,\Omega}|\gamma||x-y|^\alpha\int_0^1(1-t)^{|\gamma|-1}t^\alpha dt.
\]
The integral on the right-hand side is bounded by \(1/|\gamma|\), meaning that the item on the left is bounded by \(|\partial^{\beta+\gamma}f(x)|+[\partial^{\beta+\gamma}f]_{\alpha,\Omega}|x-y|^\alpha\). Rearranging our initial Taylor expansion and applying this inequality gives the claimed estimate with \(C=\sum_{|\gamma|=k-|\beta|}[\partial^{\beta+\gamma}f]_{\alpha,\mathbb{R}^n}/\gamma!\).
\end{proof}

Along with derivatives, we also find it useful to make H\"older estimates for powers of functions. These are enabled by the following result.

\begin{lem}\label{lem:almostLipschitz}
Let \(f\in C^{\alpha}(\mathbb{R}^n)\) be non-negative, let \(s>1\) and fix \(\varepsilon>0\). Then for any \(x,y\in\mathbb{R}^n\),
\[
    |f(x)^s-f(y)^s|\leq \frac{(1+\varepsilon)[f]_{\alpha,\mathbb{R}^n}^s}{((1+\varepsilon)^\frac{1}{s-1}-1)^{s-1}}|x-y|^{s\alpha}+\varepsilon\max\{f(x)^s,f(y)^s\}.
\]
\end{lem}

\begin{proof}
For a number \(0<\lambda<1\) to be chosen momentarily, observe that convexity gives
\[
    f(x)^s\leq\frac{1}{\lambda^s}(\lambda|f(x)-f(y)|+\lambda f(y))^s\leq \frac{1}{\lambda^{s-1}}|f(x)-f(y)|^s+\frac{1}{(1-\lambda)^{s-1}} f(y)^s.
\]
Rearranging and applying the H\"older estimate \(|f(x)-f(y)|^s\leq [f]_{\alpha,\mathbb{R}^n}^s|x-y|^{s\alpha}\) thus shows that
\[
    f(x)^s-f(y)^s\leq \frac{[f]_{\alpha,\mathbb{R}^n}^s}{\lambda^{s-1}}|x-y|^{s\alpha}+\bigg(\frac{1}{(1-\lambda)^{s-1}}-1\bigg)f(y)^s.
\]
Next select \(\lambda=1-(1+\varepsilon)^{-\frac{1}{s-1}}\), so that the coefficient on the last term is exactly \(\varepsilon\) and
\[
    f(x)^s-f(y)^s\leq \frac{(1+\varepsilon)[f]_{\alpha,\mathbb{R}^n}^s}{((1+\varepsilon)^\frac{1}{s-1}-1)^{s-1}}|x-y|^{s\alpha}+\varepsilon f(y)^s.
\]
Interchanging \(x\) and \(y\) and then taking a maximum, we obtain the desired result.
\end{proof}

Finally, we introduce some helpful notation. Given \(\xi\in S^{n-1}\), the \(k^\mathrm{th}\)-order derivative in the direction of \(\xi\) is defined by
\[
    \partial^k_\xi f=(\xi\cdot\nabla)^kf=\sum_{|\beta|=k}\frac{k!}{\beta!}\xi^\beta\partial^\beta f.
\]
Also, given \(f\in C^{k,\alpha}(\mathbb{R}^n)\) we sometimes wish to discuss all \(k^\mathrm{th}\) order derivatives simultaneously; to do this we let \(|\nabla^kf|\) denote the vector of \(k^{th}\)-order derivatives of \(f\), so that \(|\nabla^kf|\in C^\alpha(\mathbb{R}^n)\).

More derivative estimates and applications of multi-index calculus arise throughout this work in specialized settings, but those listed above suffice for most of our arguments.

\subsection*{Infinite Graphs}

Though pure graph theory may seem out of place in a work on sums of squares, our scheme for bounding \(m(n,k,\alpha)\) relies in part on estimating the chromatic number \(\chi(G)\) of an infinite graph \(G\) which we construct in the next section. For easy reference, we first quote a relevant result in this area due to Behzad \& Radjavi \cite{behzad}. We remind the reader that the degree \(\mathrm{deg}\;v\) of a vertex \(v\in V(G)\) is the number of incident edges with that vertex.

\begin{thm}[Behzad \& Radjavi, 1976]\label{basic chromatic}
Let \(G\) be a connected infinite graph with \(\mathrm{deg}\;v\leq b\) for each \(v\in V(G)\) and some \(b\in\mathbb{N}\). Then \(\chi(G)\leq b\).
\end{thm}

The proof given in \cite{behzad} involves the famous theorems of Brooks and De Bruijn \& Erd\H{o}s, and it relies only on the degree bound and the fact that \(G\) is infinite. The graph which we later construct enjoys a great deal more structure not accounted for by these arguments, and so it is reasonable to suspect that a better bound for \(\chi(G)\) is possible. To improve over Brooks' Theorem, Welsh \& Powell \cite{Welsh} worked out the following result.

\begin{thm}[Welsh \& Powell, 1967]\label{thm:WP}
Let \(H\) be a finite graph whose vertices \(v_1,\dots,v_\ell\) are ordered by decreasing degree. Then \(\chi(H)\leq \displaystyle\max_{1\leq i\leq \ell}\min\{\mathrm{deg}\;v_i+1,i\}\).
\end{thm}

This implies Brooks' Theorem, and for small graphs it is often much better. However, it loses its advantage for graphs which have a large number of vertices of high degree. If these high degree vertices are spread out, this difficulty is largely artificial. For the remainder of this section then, we adapt the proof of Theorem \ref{thm:WP} to obtain a result suitable to our infinite graph.

For \(s\in\mathbb{N}\) we say that a finite graph \(H\) has structure \(\alpha_s\) if there exist distinct vertices \(v\) and \(w_1,\dots,w_s\) such that \(\{v,w_j\}\in E(H)\) and \(\mathrm{deg}\;w_j\geq\mathrm{deg}\;v\) for each \(j=1,\dots,s\), and if whenever \(\{w_i,w_j\}\not\in E(H)\) for \(j>i\), there exists \(z\in V(H)\setminus\{v,w_1,\dots,w_s\}\) such that \(\mathrm{deg}\;z\geq\mathrm{deg}\;w_j \), \(\{w_i,z\}\not\in E(G)\), and \(\{w_j,z\}\in E(G)\).

\begin{thm}\label{best chromatic}
If a finite graph \(H\) does not have the structure \(\alpha_s\) for \(s\in\mathbb{N}\), then \(\chi(H)\leq s\).
\end{thm}

Ignoring part of the structure \(\alpha_s\), we recover a simple criterion for \(\chi(H)\) to be small.

\begin{cor}\label{better chromatic}
Let \(s\in\mathbb{N}\) and let \(H\) be a finite graph. If \(H\) does not have a vertex with \(s\) distinct neighbouring vertices of the same or larger degree, then \(\chi(H)\leq s\).
\end{cor}

The bounds we obtain for chromatic numbers of arbitrary finite induced sub-graphs of an infinite graph \(G\) also serve as bounds for \(\chi(G)\) thanks to the following well-known result.

\begin{thm}[De Bruijn \& Erd\H{o}s, 1951]\label{Erdos}
Let \(G\) be an infinite graph. Then \(\chi(G)=m\) if and only if \(\chi(H)\) is at most \(m\) for any finite subgraph of \(H\) of \(G\).
\end{thm}

The first part of our proof of Theorem \ref{best chromatic} involving greedy colouring follows \cite{Welsh} almost exactly. It is in the analysis of the output of this colouring algorithm that our approach differs.

\begin{proof}[Proof of Theorem \ref{best chromatic}]
Following \cite{Welsh}, we first perform a greedy colouring of \(H\). Let \(H\) be a finite graph, and list its vertices \(v_1,\dots, v_\ell\) in order of decreasing degree. Define \(T_1\subseteq V(H)\) by fixing \(v_1\in T_1\), and if \(v_{i_1},\dots, v_{i_t}\in T_1\) where \(1=i_1<\cdots<i_t\), assign \(v_j\) to \(T_1\) if and only if \(j\) is the smallest integer exceeding \(i_t\) for which \(v_j\) is not adjacent to any of \(\{v_{i_1},\dots,v_{i_t}\}\). Since \(H\) is finite, we can proceed until \(T_1\) is maximal.

Now form \(T_2\) by finding the largest \(j\) such that \(v_j\not\in T_1\) and let \(v_j\in T_2\). Proceed as above for vertices in \(V(H)\setminus T_1\), adding a vertex to \(T_2\) if it has maximal degree among vertices not already assigned to \(T_2\). Proceeding in this way, we see that \(T_{|V(H)|+1}=\emptyset\), meaning that this algorithm terminates in finitely many steps.

Here Welsh \& Powell estimate the smallest number \(s\in\mathbb{N}\) such that \(T_{s+1}=\emptyset\). We now do the same using different techniques. Assume that for \(s\in\mathbb{N}\) there exists \(v\in V(H)\) such that \(v\not\in T_1\cup\cdots\cup T_s\). Then \(v\) is adjacent to an element in each of the sets \(T_1,\dots,T_s\), for otherwise it would belong to one of them, and \(v\) is adjacent to a vertex of degree at least \(\mathrm{deg}\;v\) in each of the sets \(T_1,\dots, T_s\). To see this, assume toward a contradiction that for some \(j\leq s\), \(v\) is only adjacent to vertices of smaller degree than itself in \(T_j\). Then \(v\) would be assigned to \(T_j\) in the greedy colouring before those other vertices, contradicting the fact that \(v\not\in T_j\). It follows that \(v\) has distinct neighbours \(w_1,\dots,w_s\) each with degree \(\mathrm{deg}\;w_j\geq \mathrm{deg}\;v\).

Given any two neighbours \(w_i\in T_i\) and \(w_j\in T_j\) of \(v\) for \(j>i\), either \(\{w_i,w_j\}\in E(H)\), or there exists a neighbour \(z\in T_i\) of \(w_j\) with \(\mathrm{deg}\;z\geq\mathrm{deg}\;w_j\). Otherwise, \(w_j\) would have been assigned \(T_i\), contrary to our assumption. Thus \(H\) has the structure \(\alpha_s\). It follows in the contrapositive that if \(H\) omits this structure, then \(T_{s+1}=\emptyset\) and \(\chi(H)\leq s\).
\end{proof}

For the graph \(G\) which we later construct, Theorem \ref{basic chromatic} gives \(\chi(G)\leq 12\), while the refinement Corollary \ref{better chromatic} gives \(\chi(G)\leq 9\). Further improvements seem plausible, given that the argument which we later employ does not use Theorem \ref{best chromatic} in full force.

\subsection*{An Algebraic Result}

In the next section we will need to sum inequalities in such a way that all but a few terms cancel. The following lemma which enables this was proved independently with this purpose in mind, and we learned later that the same technique is used for a similar purpose in \cite[\S4]{BonyFr}.

\begin{lem}\label{thm:specialodds}
Let \(\ell\) be an odd integer and set \(s=\frac{\ell+1}{2}\). There exist numbers \(\eta_1,\dots,\eta_s\in\mathbb{R}\) and \(q_1,\dots,q_s\geq 0\) such that for each odd \(j\leq \ell\),
\[
    \sum_{i=1}^sq_i\eta_i^j=\begin{cases}
    0 & j<\ell,\\
    1 & j=\ell.
    \end{cases}
\]
\end{lem}

\begin{proof}
It suffices to show that \(Mq=e_s\) has a non-negative solution \(q=(q_1,\dots,q_s)^t\in\mathbb{R}^s\), where \(e_s\) is the last standard basis vector in \(\mathbb{R}^s\) and the square matrix \(M\) is given entry-wise by \((M)_{ij}=\eta_j^{2i-1}\). Our aim is to select \(\eta_1,\dots,\eta_s\) so that \(M\) is invertible and the entries of \(M^{-1}e_s\) are non-negative. If \(M\) is nonsingular then then \(q_k\) is the \(k^\mathrm{th}\) entry in column \(s\) of \(M^{-1}\) given by
\[
    q_k=\frac{(-1)^{s+k}\mathrm{det}(M_{sk})}{\mathrm{det}(M)},
\]
where \(M_{sk}\) is the matrix minor of \(M\) obtained by deleting row \(s\) and column \(k\) of \(M\).

Now we establish a formula for \(\mathrm{det}(M)\). This is a homogeneous polynomial in the variables \(\eta_1,\dots,\eta_s\) of degree \(s^2\), and if \(\eta_i=0\) for any \(i\) then \(M\) is singular, meaning that \(\eta_i\mid \mathrm{det}(M)\). Similarly, if \(i\neq j\) and \(\eta_i=\pm\eta_j\) then \(M\) has two linearly dependent columns and \(\mathrm{det}(M)=0\), so \((\eta_i^2-\eta_j^2)\mid \mathrm{det}(M)\) whenever \(i\neq j\). It follows that \(P\mid\mathrm{det}(M)\), where
\[
    P=\bigg(\prod_{i=1}^s\eta_i\bigg)\bigg(\prod_{i=1}^s\prod_{j=1}^{i-1}(\eta_i^2-\eta_j^2)\bigg).
\]
Since \(\mathrm{deg}\;P=s^2\), it follows from the Fundamental Theorem of Algebra that \(\mathrm{det}(M)/P\) has degree zero and is constant. Moreover, it is easy to verify that \(\mathrm{det}(M)/P=1\) and so \(\mathrm{det}(M)=P\). We can compute \(\mathrm{det}(M_{sk})\) in similar faskion, since this minor assumes the same form as \(M\). Omitting appropriate terms from the determinant formula for \(M\), we get
\[
    \mathrm{det}(M_{sk})=\bigg(\prod_{\substack{1\leq i\leq s\\
    i\neq k}}\eta_i\bigg)\bigg(\prod_{\substack{1\leq j< i\leq s\\i,j\neq k}}(\eta_i^2-\eta_j^2)\bigg).
\]
Equipped with this identity and the formula for \(q_k\) above, we are now able to write
\[
    q_k=(-1)^{s+k}\bigg(\eta_k\prod_{j=1}^{k-1}(\eta_k^2-\eta_j^2)\prod_{i=k+1}^s(\eta_i^2-\eta_k^2)\bigg)^{-1}=\bigg(\eta_k\prod_{\substack{1\leq i\leq s\\i\neq k}}(\eta_k^2-\eta_i^2)\bigg)^{-1}.
\]
Taking \(\eta_k=(-1)^{s+k}k\) ensures that each \(q_k\) is positive, and we are finished. 
\end{proof}

\section{Partition Of Unity}

Now we construct the partition of unity from which we extract the desired SOS decomposition. The key property of our construction is that the functions comprising the partition are not too closely clustered together, which allows us to form a SOS decomposition with relatively few functions. The construction is in many ways standard, taking inspiration from classical results like \cite[Lem. 1.2]{guzman}. It is modelled loosely on \cite[Lem. 5.4]{Tataru}, and it relies on a property of functions with controlled oscillation.

\begin{defn}
A function \(r:\mathbb{R}^n\rightarrow\mathbb{R}\) is said to be slowly-varying if \(r\geq 0\) everywhere, and given \(0<c<1\) there exists a constant \(\nu>0\) such that \(|r(x)-r(y)|\leq cr(x)\) whenever \(|x-y|\leq\nu r(x)\).
\end{defn}

The most obvious slowly-varying functions are the uniformly Lipschitz functions on \(\mathbb{R}^n\). A more interesting example is given in Lemma \ref{lem:slowvar}. For such functions we have the following result.

\begin{thm}\label{thm:party}
If \(r\) is slowly-varying and \(1<\lambda<\frac{3}{2}\), then there exists a family of dyadic cubes \(\{Q_j\}\) and non-negative smooth functions \(\{\psi_j\}\) such that each \(\psi_j\) is supported in \(\lambda Q_j\) and
\begin{equation}\label{eq:pou2}
    \sum_{j=1}^\infty\psi_j^2=
    \chi_{\{r>0\}}.
\end{equation}
This partition of unity has the following properties:
\begin{itemize}
    \item[(1)] No more than \(2^n\) functions in \(\{\psi_j\}\) are non-zero at any given point,
    \item[(2)] For every multi-index \(\beta\), each \(\psi_j\) satisfies the estimate \(|\partial^\beta\psi_j(x)|\leq C r(x)^{-|\beta|}\),
    \item[(3)] For every multi-index \(\beta\) and \(0<\alpha\leq 1\), each \(\psi_j\) satisfies \([\partial^\beta\psi_j]_{\alpha}(x)\leq Cr(x)^{-\alpha-|\beta|}\),
    \item[(4)] One can partition \(\{\psi_j\}\) into \(4^n-2^n\) collections of functions with pairwise-disjoint supports.
    \item[(5)] If \(n=2\), this partition of \(\{\psi_j\}\) can be achieved using only \(9\) pairwise disjoint collections.
\end{itemize}
\end{thm}

\noindent\textit{Remark}: \textit{A priori} the case \(n=2\) is not fundamentally different from higher dimensions. Rather, the argument that we use to prove property \textit{(5)} is not easy to implement when \(n\geq 3\).

\begin{proof}
First we construct an appropriate family of cubes. Let \(r\) be slowly varying with \(\nu>0\) chosen so that \(c\leq\frac{1}{5}\), and let \(\{Q_j\}\) be the collection of maximal dyadic cubes which satisfy
\[
    \ell(Q)\leq \frac{\nu}{2\sqrt{n}}\inf_{Q}r.
\]
We show that the quantities in the inequality above are in fact equivalent. Given a cube \(Q\), we denote its parent cube by \(\Tilde{Q}\). If \(y\in Q\) and \(z\in\Tilde{Q}\) then \(|y-z|\leq \sqrt{n}\ell(\Tilde{Q})=2\sqrt{n}\ell(Q)\leq \nu r(y)\), and from slow variation we can conclude that \((1-c)r(y)\leq r(z)\). Since \(y\) and \(z\) were arbitrary we find that \((1-c)\inf_Qr\leq \inf_{\tilde{Q}}r\), and using this fact with maximality of \(Q\) we get
\[
    \ell(Q)=\frac{1}{2}\ell(\tilde{Q})>\frac{\nu}{4\sqrt{n}}\inf_{\Tilde{Q}}r\geq\frac{(1-c)\nu}{4\sqrt{n}}\inf_{Q}r.
\]

Equipped with the inequalities above, we next show that if \(Q_i\) and \(Q_j\) are adjacent cubes (i.e. if \( \overline{Q_1}\cap\overline{Q_j}\neq\emptyset \)) then \(\ell(Q_i)\leq 2\ell(Q_j)\) and \(\ell(Q_j)\leq 2\ell(Q_i)\). For brevity we let \(r_i=\inf_{Q_i} r\), and we observe that if there exists \(x\in \overline{Q_i}\cap\overline{Q_j} \) then for any \(y\in Q_i\) we have \(|x-y|\leq \sqrt{n}\ell(Q_i)\). From our selection of \(Q_i\) it follows that \(|x-y|\leq\nu r_i\leq \nu r(y)\), and since \(r\) is slowly-varying we have \((1-c)r(y)\leq r(x)\leq (1+c)r(y)\). Since \(y\in Q_i\) was arbitrary, we have \((1-c)r_i\leq r(x)\leq (1+c)r_i\). The same inequalities hold if we replace \(r_i\) with \(r_j\), meaning that
\[
   \bigg(\frac{1-c}{1+c}\bigg)r_j\leq  r_i\leq \bigg(\frac{1+c}{1-c}\bigg)r_j.
\]
Finally, we combine the inequalities above to see that if the cubes \(Q_i\) and \(Q_j\) are adjacent then
\[
    \ell(Q_i)\leq \frac{\nu}{2\sqrt{n}}r_i\leq \frac{\nu(1+c)}{2\sqrt{n}(1-c)}r_j\leq\frac{2(1+c)}{(1-c)^2}\ell(Q_j)<4\ell(Q_j),
\]
where the last inequality follows from our initial selection of \(c\). Since \(Q_i\) and \(Q_j\) are dyadic, this implies that \(\ell(Q_i)\leq 2\ell(Q_j)\), and an identical argument shows that \(\ell(Q_j)\leq 2\ell(Q_i)\).

Now we fix \(1<\lambda<\frac{3}{2}\), and we study the dilated family \(\{\lambda Q_j\}\) with the aim of showing that no point in \(\mathbb{R}^n\) belongs to more than \(2^n\) dilated cubes. First, we establish that if \(\lambda Q_i\cap\lambda Q_j\neq \emptyset\) then \(Q_i\) and \(Q_j\) are adjacent. To do this we argue by contraposition, assuming that \(Q_i\) and \(Q_j\) are not adjacent. Writing each cube as a Cartesian product of intervals in the form
\[
    Q_i=\prod_{k=1}^nI_{ik},
\]
we observe that if \(d_k=\mathrm{dist}(I_{ik},I_{jk})\) then \(\mathrm{dist}(Q_i,Q_j)\geq \min\{d_1,\dots,d_n\}\). Since \(Q_i\) and \(Q_j\) are non-adjacent dyadic cubes, it follows from the inequalities for adjacent cubes established above that for each \(k=1,\dots,n\), we must have \(d_k\geq \frac{1}{2}\max\{\ell(Q_i),\ell(Q_j)\}\). This is because the distance in any direction is at least the length of some cube that is adjacent to \(Q_i\) or \(Q_j\).

Clearly \(\lambda Q_i\) is the Cartesian product of the dilated intervals, and for each \(k\) we can observe that \(\mathrm{dist}(\lambda I_{ik},\lambda I_{jk})=\mathrm{dist}(I_{ik},I_{jk})-(\frac{\lambda-1}{2})(\ell(Q_i)+\ell(Q_j))\) since \(|I_{ik}|=\ell(Q_i)\) for each \(k=1,\dots,n\). Consequently we have the estimate
\[
    \mathrm{dist}(\lambda I_{ik},\lambda I_{jk})\geq \frac{1}{2}\max\{\ell(Q_i),\ell(Q_j)\}-\bigg(\frac{\lambda-1}{2}\bigg)(\ell(Q_i)+\ell(Q_j))\geq \bigg(\frac{3}{2}-\lambda\bigg)\max\{\ell(Q_i),\ell(Q_j)\}.
\]
Since \(\lambda<\frac{3}{2}\) by assumption, we find that \(\mathrm{dist}(\lambda I_{ik},\lambda I_{jk})>0\) for each \(k\). It follows that \(\mathrm{dist}(\lambda Q_i,\lambda Q_j)>0\) meaning that \(\lambda Q_i\) and \(\lambda Q_j\) do not intersect. Hence, if \(\lambda Q_i\) and \(\lambda Q_j\) intersect, then \(Q_i\) and \(Q_j\) are adjacent as claimed.

Now fix \(x\in\mathbb{R}^n\), and relabelling if necessary let \(Q_1,\dots,Q_m\) be the cubes whose dilates by \(\lambda\) contain \(x\). If \(Q_i\) and \(Q_j\) are any two of these cubes, then \(x\in \lambda Q_i\cap\lambda Q_j\) and it follows from the argument above \(Q_i\) and \(Q_j\) are adjacent, so the cubes \(Q_1,\dots,Q_m\) are pairwise adjacent. We use this fact to obtain an upper bound on \(m\), by showing that \(\overline{Q_1}\cap\cdots\cap \overline{Q_m}\neq\emptyset\). To this end we once again write each cube as a Cartesian product of closed intervals, and use the fact that Cartesian products and intersections commute:
\[
    \bigcap_{i=1}^m\overline{Q_i}=\bigcap_{i=1}^m\prod_{k=1}^nI_{ik}=\prod_{k=1}^n\bigcap_{i=1}^mI_{ik}.
\]
The second identity is easily verified by induction. For each fixed \(k\), now consider the intersection
\[
    \bigcap_{i=1}^mI_{ik}.
\]
Since the family of cubes is pairwise adjacent, any two intervals \(I_{ik}\) and \(I_{jk}\) must intersect, for otherwise \(\mathrm{dist}(Q_i,Q_j)>0\). Writing \(I_{ik}=[a_{ik},b_{ik}]\), we take \(a_k=\max_i\{a_{ik}\}\) and \(b_k=\min_i\{b_{ik}\}\) and we note that \(a_k\leq b_k\) for if not, then for some \(i\) and \(j\) we have \(a_{ik}>b_{jk}\) and \(I_{ik}\cap I_{jk}=\emptyset\). Since \([a_k,b_k]\subseteq I_{ik}\) for each \(i=1\dots,m\) we have that
\[
    \bigcap_{i=1}^mI_{ik}\supseteq[a_k,b_k]\neq\emptyset,
\]
since the interval contains at least the point \(a_k\). It follows that \(\overline{Q_1}\cap\cdots\cap \overline{Q_m}\) is a Cartesian product of non-empty sets, hence it is nonempty. Since a given point in \(\mathbb{R}^n\) can belong to the boundary of at most \(2^n\) closed cubes, we conclude that \(m\leq 2^n\). Finally, since \(x\in\mathbb{R}^n\) was arbitrary, we conclude that any point can belong to at most \(2^n\) dilated cubes.

There is an additional property of this family \(\{Q_j\}\) that will be useful momentarily. Namely, we are able to write
\begin{equation}\label{eq:setseq}
    S=\{x\in\mathbb{R}^n:\;r(x)>0\}=\bigcup_{j=1}^\infty\lambda Q_j.   
\end{equation}
Clearly \(S\) is contained in the union of dilated cubes, and to verify the reverse inclusion we fix \(x\in\lambda Q_j\) for any \(j\). Since \(\lambda<\frac{3}{2}\), it follows that for any \(y\in Q_j\) we have that \(r(y)>0\) and
\[
    |x-y|\leq \frac{(1+\lambda)\sqrt{n}}{2}\ell(Q_j)\leq \frac{(1+\lambda)\nu}{4}\inf_{Q_j}r\leq \nu r(y).
\]
From slow variation of \(r\) this implies that \(0<(1-c)r(y)\leq r(x)\), allowing us to conclude that \(x\in S\). Since \(x\) was any member of the union, we see that \eqref{eq:setseq} holds.

Equipped with the family of dyadic cubes constructed above, we now define the functions \(\{\psi_j\}\) by using a construction resembling that in \cite{Tataru}. First let \(\phi:\mathbb{R}\rightarrow\mathbb{R}\) be a smooth non-negative function supported in \([-\frac{\lambda}{2},\frac{\lambda}{2}]\), for which \(\phi=1\) on \([-\frac{1}{2},\frac{1}{2}]\). Fixing a cube \(Q\), we denote its center point by \(\tilde{x}\) and we set
\[
    \psi_Q(x)=\prod_{k=1}^n\phi\bigg(\frac{\tilde{x}_k-x_k}{\ell(Q)}\bigg).
\]
In this way, we see that \(\psi_Q=1\) on \(Q\), and that \(\psi_Q\) is supported in \(\lambda Q\). Moreover, for any multi-index \(\beta=(\beta_1,\dots,\beta_n)\) we observe that thanks to smoothness of \(\phi\), the following estimate holds for some constant \(C\) which does not depend on \(Q\),
\[
    |\partial^\beta\psi_Q(x)|=\prod_{k=1}^n\bigg|\frac{\partial^{\beta_k}}{\partial x_k^{\beta_k}}\phi\bigg(\frac{\tilde{x}_k-x_k}{\ell(Q)}\bigg)\bigg|\leq C\prod_{k=1}^n\ell(Q)^{-\beta_k}=C\ell(Q)^{-|\beta|}.
\]
This estimate will be useful momentarily. Finally, we can define the family of functions \(\{\psi_j\}\) by taking the dyadic collection \(\{Q_j\}\) constructed above and for each \(j\) setting
\[
    \psi_j=\psi_{Q_j}\bigg(\sum_{i=1}^\infty\psi_{Q_i}^2\bigg)^{-\frac{1}{2}}.
\]
There is no issue of convergence inside \(S\) since at least one and at most \(2^n\) terms in the sum above are nonzero at any given point. It is also clear by the definition that if \(x\in S\) then
\[
    \sum_{j=1}^\infty\psi_j^2=1.
\]
Since each \(\psi_{Q_j}\) is supported in \(S\) (as \(\lambda Q_j\subset S\) for each \(j\), as we showed above) so too is \(\psi_j\). Hence we can smoothly extend each \(\psi_j\) by zero outside of \(S\) to get equation \eqref{eq:pou2}. Additionally, given \(x\in\mathbb{R}^n\), we recall that \(x\) is in at most \(2^n\) of the dilates \(\{\lambda Q_j\}\), meaning that at most \(2^n\) of the functions \(\{\psi_j\}\) can be non-zero at any given point and property \textit{(1)} holds.

Now we establish estimates \textit{(2)} and \textit{(3)}. For the former, we fix \(j\in\mathbb{N}\) and conduct a straightforward (albeit tedious) calculation using the generalized product and chain rules discussed in Section 2 to see that for any multi-index \(\beta\),
\[
    \partial^\beta\psi_j=\sum_{\gamma\leq\beta}\sum_{\Gamma\in P(\gamma)}C_{\gamma,\Gamma}(\partial^{\beta-\gamma}\psi_{Q_j})\bigg(\sum_{i=1}^\infty\psi_{Q_i}^2\bigg)^{-\frac{1}{2}-|\Gamma|}\prod_{\eta \in\Gamma}\sum_{i=1}^\infty\sum_{\delta\leq\eta}\binom{\eta}{\delta}(\partial^{\eta-\delta}\psi_{Q_i})(\partial^\delta\psi_{Q_i}).
\]
If \(x\not\in \overline{\lambda Q_j}\) then \(\partial^\beta\psi_j(x)=0\), and \textit{(2)} holds trivially. On the other hand, if \(x\in \overline{\lambda Q_j}\) then at most \(2^n\) of the functions \(\psi_{Q_i}\) and their derivatives are not identically zero in a neighbourhood of \(x\). Re-indexing if necessary, we call these functions \(\psi_{Q_1},\dots,\psi_{Q_{2^n}}\) and we note that \(\ell(Q_j)\leq 2\ell(Q_i)\) for each \(i=1,\dots,2^n\). Recalling the estimate \(|\partial^\beta\psi_{Q_i}|\leq C\ell(Q_i)^{-|\beta|}\), we find that
\[
    \prod_{\eta\in\Gamma}\bigg|\sum_{i=1}^\infty\sum_{\delta\leq\eta}\binom{\eta}{\delta}(\partial^{\eta-\delta}\psi_{Q_i})(\partial^\delta\psi_{Q_i})\bigg|\leq C\prod_{\eta\in\Gamma}\sum_{i=1}^{2^n}\sum_{\delta\leq\eta}\ell(Q_i)^{-|\eta|}\leq C\prod_{\eta\in\Gamma}\ell(Q_j)^{-|\eta|}=C\ell(Q_j)^{-|\gamma|}.
\]
Since \(\psi_{Q_i}(x)=1\) for some \(i\in\mathbb{N}\) we also have
\[
    \bigg(\sum_{i=1}^\infty\psi_{Q_i}^2\bigg)^{-\frac{1}{2}-|\Gamma|}\leq 1.    
\]
Employing these pointwise estimates in the formula for \(\partial^\beta\psi_j\) stated above, we obtain the bound
\[
    |\partial^\beta\psi_j|\leq C\sum_{\gamma\leq\beta}\sum_{\Gamma\in P(\gamma)}|\partial^{\beta-\gamma}\psi_{Q_j}|\ell(Q_j)^{-|\gamma|}\leq C\sum_{\gamma\leq\beta}\sum_{\Gamma\in P(\gamma)}\ell(Q_j)^{|\gamma|-|\beta|}\ell(Q_j)^{-|\gamma|}=C\ell(Q_j)^{-|\beta|}.
\]
Finally, we use the fact established above that \(\ell (Q_j)\) is comparable to \(r(x)\) on \(\lambda Q_j\) to conclude that  \(|\partial^\beta\psi_j|\leq Cr^{-|\beta|}\), giving property \textit{(2)}.

For property \textit{(3)} we note that the estimate is trivial when \(x\not\in\overline{\lambda Q_j}\), since \([\partial^\beta \psi_j]_\alpha(x)=0\). If \(x\in \overline{\lambda Q_j}\) then from the estimate \textit{(2)} and the mean value theorem, we have for \(y,z\in\overline{\lambda Q_j}\) that
\[
    |\partial^\beta\psi_j(y)-\partial^\beta\psi_j(z)|\leq C\ell(Q_j)^{-|\beta|-1}|y-z|\leq C\ell(Q_j)^{-|\beta|-\alpha}|y-z|^{\alpha}.
\]
From this estimate, it follows that \([\partial^\beta \psi_j]_\alpha(x)\leq C\ell(Q_j)^{-|\beta|-\alpha}\). Property \textit{(3)} then comes as a consequence of the equivalence of \(r\) with \(\ell(Q_j)\) on \(\lambda Q_j\).

To prove property \textit{(4)}, we form an infinite graph \(G\) as follows. Associate to each cube in \(\{Q_j\}\) a vertex \(v_j\), and add edge \(\{v_i,v_j\}\) if the corresponding cubes \(Q_i\) and \( Q_j\) are adjacent (i.e. if \(\overline{Q_i}\cap\overline{Q_j}\neq\emptyset\)). It follows from continuity of \(r\) that \(G\) is an infinite graph, and we can assume without loss of generality that \(G\) is connected by treating each connected component of the set \(S\) separately. We wish to bound \(\chi(G)\), the chromatic number of \(G\).

First we do this in any number of dimensions, before specializing to \(n=2\). Each cube in \(\{Q_j\}\) is adjacent to at most \(4^n-2^n\) other cubes, since adjacent cubes satisfy \(\ell(Q_i)\leq 2\ell(Q_j)\), and so \(d=\displaystyle\sup_{v\in V(G)}\mathrm{deg}\;v\leq 4^n-2^n.\) It follows from Theorem \ref{basic chromatic} that \(\chi(G)\leq 4^n-2^n\), giving \textit{(4)}.

Finally, to prove \textit{(5)} we fix \(n=2\) and we consider an arbitrary finite induced subgraph \(H\) of \(G\). Fix \(s=9\) and assume toward a contradiction that there exists \(v\in V(H)\) with \(s\) neighbours each with degree at least \(s\). In terms of the dyadic grid, this means that there exists a cube \(Q\) which is adjacent to at least \(9\) cubes \(Q_1,\dots,Q_9\), which are each adjacent to at least \(9\) cubes. At most four of these cubes intersect \(\overline{Q}\) only at a corner, and the remaining cubes which we relabel as \(Q_1,\dots, Q_5\) must share a face with \(Q\). It follows that \(\ell(Q_1)+\cdots+\ell(Q_{5})\leq 4\ell(Q),\) meaning that at least two neighbours (say \(Q_1\) and \(Q_2\)) have size \(\frac{1}{2}\ell(Q)\). This implies that \(Q_1\) and \(Q_2\) are adjacent to at most \(8\) other cubes: \(Q\) and three of its neighbours, and at most four adjacent cubes with length at least \(\frac{1}{4}\ell(Q)\) opposite \(Q\).

If \(\mathrm{deg}\;v\leq 10\) then we are finished, for this implies that at most \(8\) neighbours of \(Q\) may induce vertices of degree at least \(s=9\). If there are \(m=11\) or \(m=12\) neighbours then four must be at corners, and two others have length \(\frac{1}{2}\ell(Q)\) and occupy at least one face of \(Q\). For the remaining neighbours which we relabel \(Q_1,\dots,Q_{m-6}\) we have
\[
    \ell(Q_1)+\cdots+\ell(Q_{5})\leq \ell(Q_1)+\cdots+\ell(Q_{m-6})=3\ell(Q),
\]
meaning that at least four of the cubes above have length \(\frac{1}{2}\ell(Q)\). Each can have at most 8 neighbours, so at least 6 neighbours of \(Q\) have fewer than \(s\) neighbours, meaning that only \(m-6\leq 6\) neighbours of \(Q\) can induce vertices of degree at least \(s\). If follows that \(G\) contains no vertices with \(s=9\) neighbours each of degree exceeding \(\mathrm{deg}\;v\geq s\). By Corollary \ref{better chromatic} we conclude that \(\chi(H)\leq 9\). Since \(H\) was arbitrary, it follows that \(\chi(G)\leq 9\) by Theorem \ref{Erdos}.
\end{proof}

Property \textit{(5)} cannot be improved using the pigeonholing argument above, since a given cube \(Q\) in \(\mathbb{R}^2\) can have 8 neighbouring vertices of degree at least \(8\). In higher dimensions, the pigeonhole argument also becomes difficult since bounding vertex degrees is challenging.

\section{Local Decompositions}

This section shows that the localization of a non-negative function \(f\in C^{k,\alpha}(\mathbb{R}^n)\) to a suitable cube \(Q\) either has a square root in \eqref{halfreg}, or it is possible to write \(f=g^2+F\) on \(Q\), where \(g\) belongs to \eqref{halfreg} and \(F\in C^{k,\alpha}(\mathbb{R}^{n-1})\). These local decompositions can be performed in the support sets of the partition functions \(\psi_j\), and they always exist when \(k\leq 3\). If \(k\geq4\) then \(f\) can be decomposed locally when some additional hypotheses are satisfied.

Throughout this section we fix non-negative \(f\in C^{k,\alpha}(\mathbb{R}^n)\) for \(k\geq 2\) and \(0<\alpha\leq 1\). Our primary instrument is a function \(r\) that controls the derivatives of \(f\). It was originally identified by Fefferman \& Phong \cite{Fefferman-Phong} in the case \(k=3\) and \(\alpha=1\), and modified by Korobenko \& Sawyer \cite{SOS_I} to work for \(k=4\) and \(0<\alpha\leq1\). For our purposes a suitable control function is given by
\begin{equation}\label{eq:controlfunc}
    r(x)=\max_{\substack{0\leq j\leq k,\\j\;\mathrm{even}}}\sup_{|\xi|=1}[\partial^j_\xi f(x)]_+^\frac{1}{k-j+\alpha}.
\end{equation}

\begin{thm}\label{thm:cauchylike}
If \(f\in C^{k,\alpha}(\mathbb{R}^n)\) is non-negative, then for every \(\ell\leq k\) there exists a constant \(C\) such that \(|\nabla^\ell f(x)|\leq Cr(x)^{k-\ell+\alpha}\) pointwise in \(\mathbb{R}^n\).
\end{thm}

\begin{proof}
For even \(\ell\leq k\), we begin by bounding the negative part of the \(\ell^\mathrm{th}\)-order directional derivatives of \(f\). To do this we use non-negativity of the Taylor polynomial of \(f\) to obtain the following inequality for \(\lambda\in\mathbb{R}\) and \(\mathbb{\xi}\in S^{n-1}\):
\[
    0\leq f(x+\lambda\xi)\leq  \sum_{j=0}^k\frac{\lambda^j}{j!}\partial^j_\xi f(x)+\frac{1}{k!}[\nabla^kf]_{\alpha,\mathbb{R}^n}|\lambda|^{k+\alpha}.
\]
A similar estimate holds if \(\lambda\) is replaced with \(-\lambda\), and summing the resulting inequalities gives
\[
    0\leq \sum_{\substack{0\leq j\leq k,\\ j\;\textrm{even}}}\frac{\lambda^j}{j!}\partial^j_\xi f(x)+\frac{2}{k!}[\nabla^kf]_{\alpha,\mathbb{R}^n}\lambda^{k+\alpha}
\]
for \(\lambda>0\). Rearranging the inequality above and noting that \(\partial^j_\xi f(x)\leq [\partial^j_\xi f(x)]_+\), we find that
\[
    -\partial_\xi^\ell f(x)\leq\sum_{\substack{0\leq j\leq k,\\ j\;\textrm{even},\;j\neq\ell}}\frac{\ell!}{j!}\lambda^{j-\ell}[\partial^j_\xi f(x)]_++\frac{2\ell!}{k!}[\nabla^kf]_{\alpha,\mathbb{R}^n}\lambda^{k-\ell+\alpha}.
\]
Using \(\lambda=\!\!\displaystyle\max_{\substack{0\leq j\leq k,\\ j\;\textrm{even},\;j\neq\ell}}[\partial^j_\xi f(x)]_+^\frac{1}{k-j+\alpha}\) now gives \(\lambda^{j-\ell}[\partial^j_\xi f(x)]_+\leq \lambda^{k-\ell+\alpha}\) and \(-\partial_\xi^\ell f(x)\leq C\lambda^{k-\ell+\alpha}\).

Taking a supremum over all \(\xi\in S^{n-1}\) and using equivalence of the norms \(|\nabla^kf(x)|\) and \(\displaystyle\sup_{|\xi|=1}|\partial^k_\xi f(x)|\), we see now that for even values of \(\ell\) the required estimate holds pointwise:
\[
    |\nabla^\ell f(x)|\leq C\sup_{|\xi|=1}[\partial_\xi^\ell f(x)]_++C\max_{\substack{0\leq j\leq k,\\ j\;\textrm{even},\;j\neq\ell}}\sup_{|\xi|=1}[\partial^j_\xi f(x)]_+^\frac{k-\ell+\alpha}{k-j+\alpha}\leq Cr(x)^{k-\ell+\alpha}.
\]

For derivatives of odd order we again begin by using a Taylor expansion and non-negativity of \(f\) to see that if \(\lambda\in\mathbb{R}\) and \(|\xi|=1\) then
\[
    0\leq \sum_{j=0}^k\frac{\lambda^j}{j!}\partial^j_\xi f(x)+\frac{1}{k!}[\nabla^kf]_{\alpha,\mathbb{R}^n}|\lambda|^{k+\alpha}.
\]
This inequality continues to hold if we replace \(\lambda\) with \(\eta\lambda\) for any \(\eta\in\mathbb{R}\). First let \(\ell\) be the largest odd number that is less than or equal to \(k\) and set \(s=\frac{\ell+1}{2}\). Replacing  \(\lambda\) with \(\eta_1\lambda\) through \(\eta_s\lambda\) for some real numbers \(\eta_1,\dots,\eta_s\) to be chosen momentarily, and adding the resulting inequalities scaled by positive constants \(q_1,\dots,q_s\), we find that
\[
    0\leq \sum_{j=0}^k\bigg(\sum_{i=1}^sq_i\eta_i^j\bigg)\frac{\lambda^j}{j!}\partial^j_\xi f(x)+C\lambda^{k+\alpha}.
\]
The preceding inequality continues to hold if we replace \(\lambda\) with \(-\lambda\), and so if \(\lambda>0\) then
\[
    \bigg|\sum_{\substack{1\leq j\leq \ell,\\ j\;\textrm{odd}}}\bigg(\sum_{i=1}^sq_i\eta_i^j\bigg)\frac{\lambda^j}{j!}\partial^j_\xi f(x)\bigg|\leq C\bigg(\sum_{\substack{0\leq j\leq k,\\ j\;\textrm{even}}}\lambda^j\partial^j_\xi f(x)+\lambda^{k+\alpha}\bigg).
\]
For brevity we denote by \(F_\lambda(x)\) the right-hand side above. Using \Cref{thm:specialodds}, choose \(\eta_1,\dots,\eta_s\in\mathbb{R}\) and \(q_1,\dots,q_s\geq 0\) so that \(\sum_{i=1}^sq_i\eta_i^j=0\) for every odd \(j<\ell\) and \(\sum_{i=1}^sq_i\eta_i^\ell=1\). With these selections, the estimate above reads \(|\lambda^\ell\partial^\ell_\xi f(x)|\leq C F_\lambda(x)\).

Equipped with this estimate, we can repeat the argument above using different constants \(\tilde{q}_1,\dots,\tilde{q}_{s-1}\geq 0\) and \(\tilde{\eta}_1,\dots\tilde{\eta}_{s-1}\in\mathbb{R}\), and by combining non-negative Taylor polynomials we get
\[
    \bigg|\sum_{\substack{1\leq j\leq \ell-2,\\ j\;\textrm{odd}}}\bigg(\sum_{i=1}^{s-1}\tilde{q}_i\tilde{\eta}_i^j\bigg)\frac{\lambda^j}{j!}\partial^j_\xi f(x)\bigg|\leq C|\lambda^\ell\partial^\ell_\xi f(x)|+CF_\lambda(x)\leq CF_\lambda(x).\\
\]
Using \Cref{thm:specialodds} again, we can choose the constants \(\tilde{q}_i\) and \(\tilde{\eta}_i\) so that the left-hand side above reduces to \(|\lambda^{\ell-2}\partial^{\ell-2}_\xi f(x)|\). This gives \(|\lambda^{\ell-2}\partial^{\ell-2}_\xi f(x)|\leq CF_\lambda(x)\), and we see by repeating this argument that for each odd \(\ell\leq k\) and \(\lambda>0\),
\[
    |\partial_\xi^\ell f(x)|\leq  CF_\lambda(x)\leq  C\bigg(\sum_{\substack{0\leq j\leq k,\\ j\;\textrm{even}}}\lambda^{j-\ell}[\partial^j_\xi f(x)]_++\lambda^{k-\ell+\alpha}\bigg).
\]
Taking \(\lambda=\displaystyle\max_{\substack{0\leq j\leq k,\\ j\;\textrm{even}}}|\partial_\xi^jf(x)|^\frac{1}{k-j+\alpha}\) in this estimate gives \(|\partial_\xi^\ell f(x)|\leq C\lambda^{k-\ell+\alpha}\), and it follows that
\[
    |\nabla^\ell f(x)|\leq C\sup_{|\xi|=1}|\partial^\ell_\xi f(x)|\leq C\sup_{|\xi|=1}\lambda^{k-\ell+\alpha}\leq Cr(x)^{k-\ell+\alpha}.
\]
Thus the claimed derivative estimates hold for every \(\ell\leq k\) as claimed.
\end{proof}

Following \cite{SOS_I,Tataru}, we now show that \(r\) is slowly-varying.

\begin{lem}\label{lem:slowvar}
There exists a constant \(\nu>0\) such that \(|r(x)-r(y)|\leq \frac{1}{4}r(x)\) when \(|x-y|\leq \nu r(x)\).
\end{lem}

\textit{Remark}: The choice of the constant \(\frac{1}{4}\) in the preceding lemma is largely arbitrary; any constant between \(0\) and \(1\) will suffice, and we only choose \(\frac{1}{4}\) to simplify some later calculations. Throughout this section we continue to put size restrictions on \(\nu\). These ensure that in the end, it is a small positive constant whose value is inconsequential for our final construction. We encounter no issue in occasionally asking that \(\nu\) be smaller than previously assumed.

\begin{proof}
Given \(x,y\in\mathbb{R}^n\), it follows from the definition of \(r\) and straightforward properties of the maximum and supremum that
\begin{equation}\label{eq:intermed4}
    |r(x)-r(y)|\leq \max_{\substack{0\leq j\leq k,\\j\;\mathrm{even}}}\big\{\sup_{|\xi|=1}\big|[\partial^j_\xi f(x)]_+^\frac{1}{k-j+\alpha}-[\partial^j_\xi f(y)]_+^\frac{1}{k-j+\alpha}\big|\big\}.
\end{equation}
Now we assume that \(|x-y|\leq \nu r(x)\). If \(j<k\) then \(k-j+\alpha> 1\), and a concavity estimate gives
\[
    \big|[\partial^j_\xi f(x)]_+^\frac{1}{k-j+\alpha}-[\partial^j_\xi f(y)]_+^\frac{1}{k-j+\alpha}\big|\leq |\partial^j_\xi f(x)-\partial^j_\xi f(y)|^\frac{1}{k-j+\alpha}.
\]
To estimate the right-hand side above, we expand the directional operator \(\partial_\xi^j\) to write
\[
    |\partial^j_\xi f(x)-\partial^j_\xi f(y)|=\bigg|\sum_{|\beta|=j}\frac{\beta!}{j!}\xi^\beta(\partial^\beta f(x)-\partial^\beta f(y))\bigg|\leq \sum_{|\beta|=j}\frac{\beta!}{j!}|\partial^\beta f(x)-\partial^\beta f(y)|.
\]
Then, \Cref{lem:Taylorest} together with the bounds \(|\nabla^\ell f(x)|\leq Cr(x)^{k-\ell+\alpha}\) and \(|x-y|\leq \nu r(x)\) gives us
\[
    |\partial^\beta f(x)-\partial^\beta f(y)|\leq C\sum_{1\leq |\gamma|\leq k-|\beta|}\nu^{|\gamma|}r(x)^{k-|\beta|+\alpha}+C\nu^{k-|\beta|+\alpha}r(x)^{k-|\beta|+\alpha}.
\]
Every power on \(\nu\) above is at least one, since \(|\beta|<k\). Thus, assuming that \(\nu\leq 1\), we find that there exists a constant \(C\) for which
\begin{equation}\label{eq:omegaeqn}
    |\partial^\beta f(x)-\partial^\beta f(y)|\leq C\nu r(x)^{k-|\beta|+\alpha}.
\end{equation}
Therefore \(|\partial^j_\xi f(x)-\partial^j_\xi f(y)|\leq C\nu r(x)^{k-|\beta|+\alpha}\), and since this bound is independent of \(\xi\) we can choose a small number \(\nu\) which is independent of \(x\) such that the following holds:
\[
    \sup_{|\xi|=1}\big|[\partial^j_\xi f(x)]_+^\frac{1}{k-j+\alpha}-[\partial^j_\xi f(y)]_+^\frac{1}{k-j+\alpha}\big|\leq  C\nu^\frac{1}{k-j+\alpha}r(x)\leq \frac{1}{4}r(x).
\]

It remains to consider the case \(j=k\), which arises when \(k\) is even. This time we employ \Cref{lem:almostLipschitz}, taking \(\beta=\frac{1}{\alpha}\) and \(\varepsilon=\nu^\alpha\). Combining with the inequality \(|x-y|\leq \nu r(x)\), we get
\[
    \sup_{|\xi|=1}\big|[\partial^k_\xi f(x)]_+^\frac{1}{\alpha}-[\partial^k_\xi f(y)]_+^\frac{1}{\alpha}\big|\leq\frac{C\nu (1+\nu^\alpha)r(x)}{((1+\nu^\alpha)^{\frac{\alpha}{1-\alpha}}-1)^\frac{1-\alpha}{\alpha}}+\nu^\alpha\max\{r(x),r(y)\}.
\]
If \(\nu\leq 1\) the quotient above is bounded by \(C\nu^\alpha r(x)\), for a constant \(C\) that depends only on \(\alpha\) and the H\"older semi-norms of \(f\). Since \(\max\{r(x),r(y)\}\leq r(x)+|r(x)-r(y)|\) by non-negativity of \(r\), when \(\nu\) is small enough we  have
\[
    \sup_{|\xi|=1}\big|[\partial^k_\xi f(x)]_+^\frac{1}{\alpha}-[\partial^k_\xi f(y)]_+^\frac{1}{\alpha}\big|\leq C\nu^\alpha r(x)+\nu^\alpha|r(x)-r(y)|\leq \frac{1}{8}r(x)+\frac{1}{2}|r(x)-r(y)|.
\]
It follows from these bounds and \eqref{eq:intermed4} that \(|r(x)-r(y)|\leq \frac{1}{4}r(x)\), as we wished to show.
\end{proof}

The following consequence of estimate \eqref{eq:omegaeqn} will also be useful momentarily.

\begin{cor}\label{cor:supplementary}
Let \(f\in C^{k,\alpha}(\mathbb{R}^n)\) be non-negative for \(k\geq 2\), and let \(r\) be as in \eqref{eq:controlfunc}. If \(\nu\) is a small positive constant, then there exists \(\omega>0\) for which
\begin{itemize}
    \item[(1)] \(|f(x)-f(y)|\leq \frac{1}{2}\omega \nu r(x)^{k+\alpha}\) whenever \(|x-y|\leq \nu r(x)\),
    \item[(2)] \(|\partial^\beta f(x)-\partial^\beta f(y)|\leq \frac{1}{2}r(x)^{k-2+\alpha}\) whenever \(|x-y|\leq \sqrt{\nu^2+6\omega\nu}r(x)\) and \(|\beta|=2\).
\end{itemize}
\end{cor}

\begin{proof}
Estimate \textit{(1)} follows by taking \(\omega=2C\) for \(C\) as in \eqref{eq:omegaeqn} when \(|\beta|=0\), while \textit{(2)} is given by replacing \(\nu\) with \(\sqrt{\nu^2+6\omega\nu}\) in \eqref{eq:omegaeqn} when \(|\beta|=2\), and once again choosing \(\nu\) small.
\end{proof}

For the remainder of this section, we fix a maximal dyadic cube \(Q\) which satisfies
\[
    \ell(Q)\leq \frac{\nu}{2\sqrt{n}}\inf_{Q}r.
\]
Letting \(x\) denote the center of \(Q\) and \(r_Q=\inf_Qr\), we consider two cases: when \(f(x)\geq \omega\nu r_Q^{k+\alpha}\), and when \(f(x)<\omega\nu r_Q^{k+\alpha}\). In the first case, we show that \(f\) has a half-regular root on \(2Q\).

\begin{lem}\label{lem:local1}
Let \(f\), \(r\), and \(\nu\) be as above and assume that \(f(x)\geq \omega\nu r_Q^{k+\alpha}\). Then for any multi-index \(\beta\) of order \(|\beta|\leq \frac{k+\alpha}{2}\), and for any \(y\in 2Q\), the following estimates hold:
\begin{itemize}
    \item[(1)] \(|\partial^\beta\sqrt{f(y)}|\leq Cr(y)^{\frac{k+\alpha}{2}-|\beta|}\),
    \item[(2)] \([\partial^\beta\sqrt{f}]_{\frac{\alpha}{2}}(y)\leq Cr(y)^{\frac{k}{2}-|\beta|}\) if \(k\) is even,
    \item[(3)] \([\partial^\beta\sqrt{f}]_{\frac{1+\alpha}{2}}(y)\leq Cr(y)^{\frac{k-1}{2}-|\beta|}\) if \(k\) is odd.
\end{itemize} 
\end{lem}

\begin{proof}
Using Lemma \ref{lem:genchain} we bound derivatives of \(f\) at \(y\in 2Q\), estimating pointwise to first get
\[
    |\partial^\beta\sqrt{f(y)}|=\bigg|\sum_{\Gamma\in P(\beta)}C_{\beta,\Gamma}f(y)^{\frac{1}{2}-|\Gamma|}\prod_{\gamma\in\Gamma}\partial^\gamma f(y)\bigg|\leq C\sum_{\Gamma\in P(\beta)}f(y)^{\frac{1}{2}-|\Gamma|}\prod_{\gamma\in\Gamma}|\partial^\gamma f(y)|.
\]
If \(y\in 2Q\) then \(|x-y|\leq \nu r_Q\), meaning that \(f(y)\geq \frac{1}{2}\omega\nu r_Q^{k+\alpha}\) by property \textit{(1)} of Corollary \ref{cor:supplementary}. Additionally, Theorem \ref{thm:cauchylike} and the fact that \(r(y)\leq Cr_Q\) for \(y\in 2Q\) together show that
\[
    |\partial^\beta\sqrt{f(y)}|\leq C\sum_{\Gamma\in P(\beta)}r_Q^{(k+\alpha)(\frac{1}{2}-|\Gamma|)}\prod_{\gamma\in\Gamma}r(y)^{k+\alpha-|\gamma|}\leq Cr(y)^{\frac{k+\alpha}{2}-|\beta|}.
\]

For the semi-norm estimates \textit{(2)} and \textit{(3)}, we treat the odd and even cases simultaneously by letting \(\Lambda\) denote the integer part of \(\frac{k}{2}\) and setting \(\lambda=\frac{k+\alpha}{2}-\Lambda\). To prove the claimed results it suffices to show that \([\partial^\beta\sqrt{f}]_{\lambda}(y)\leq Cr(y)^{\Lambda-|\beta|}\) whenever \(|\beta|\leq \Lambda\) and \(y\in 2Q\). Given \(\beta\in\mathbb{N}_0^n\) observe that by Lemma \ref{lem:genchain} and the triangle inequality, we have for \(w,z\in 2Q\) that
\begin{equation}\label{eq:ctrl}
\begin{split}
    |\partial^\beta\sqrt{f(w)}-\partial^\beta\sqrt{f(z)}|&\leq C\sum_{\Gamma\in P(\beta)} f(w)^{\frac{1}{2}-|\Gamma|}\bigg|\prod_{\gamma\in\Gamma}\partial^\gamma f(w)-\prod_{\gamma\in\Gamma}\partial^\gamma f(z)\bigg|\\
    &\qquad+C\sum_{\Gamma\in P(\beta)} |f(w)^{\frac{1}{2}-|\Gamma|}-f(z)^{\frac{1}{2}-|\Gamma|}|\prod_{\gamma\in\Gamma}|\partial^\gamma f(z)|.    
\end{split}
\end{equation}
We treat the terms on the right-hand side above separately, first using the mean value theorem and slow variation of \(r\) on \(2Q\) to obtain the bound
\[
    |f(w)^{\frac{1}{2}-|\Gamma|}-f(z)^{\frac{1}{2}-|\Gamma|}|\leq Cr(y)^{\frac{k+\alpha}{2}-1-|\Gamma|(k+\alpha)}|w-z|\leq Cr(y)^{\frac{k+\alpha}{2}-|\Gamma|(k+\alpha)-\lambda}|w-z|^\lambda.
\]
Explicitly, we have used here that \(f\) is bounded below by a multiple of \(r(y)^{k+\alpha}\) on \(2Q\), followed by the fact that \(|w-z|\leq Cr(y)^{1-\lambda}|w-z|^\lambda\) for \(w,z\in 2Q\). Bounding the other derivatives of \(f\) on \(2Q\) in the same fashion, we see that
\[
    |f(w)^{\frac{1}{2}-|\Gamma|}-f(z)^{\frac{1}{2}-|\Gamma|}|\prod_{\gamma\in\Gamma}|\partial^\gamma f(z)|\leq Cr(y)^{\frac{k+\alpha}{2}-|\Gamma|(k+\alpha)-\lambda}|w-z|^\lambda\prod_{\gamma\in\Gamma}r(y)^{k+\alpha-|\gamma|}.
\]
Since \(\Gamma\in P(\beta)\), we can simplify the right-hand side above to \(Cr(y)^{\Lambda-|\beta|}|w-z|^\lambda\). To estimate the remaining term of \eqref{eq:ctrl}, we observe that
\begin{equation}\label{eq:sub}
    \bigg|\prod_{\gamma\in\Gamma}\partial^\gamma f(w)-\prod_{\gamma\in\Gamma}\partial^\gamma f(z)\bigg|\leq \sum_{\gamma\in\Gamma}\bigg(\prod_{\mu\in\Gamma\setminus\{\gamma\}}\sup_{B}|\partial^\mu f|\bigg)|\partial^\gamma f(w)-\partial^\gamma f(z)|.    
\end{equation}

Arguing as above, we can control the product in \eqref{eq:sub} by \(Cr(y)^{(k+\alpha)(|\Gamma|-1)-|\beta|+|\gamma|}\). Further, since \(|\beta|\leq\Lambda< \frac{k+\alpha}{2}\), whenever \(\gamma\in P(\beta)\) we must have \(|\gamma|<k\), meaning that we can use the mean value theorem and Theorem \ref{thm:cauchylike} to get for \(w,z\in 2Q\) that
\[
    |\partial^\gamma f(w)-\partial^\gamma f(z)|\leq Cr(y)^{k+\alpha-|\gamma|-\lambda}|w-z|^\lambda.
\]
Therefore the left-hand side of \eqref{eq:sub} is bounded by \(Cr(y)^{(k+\alpha)|\Gamma|-|\beta|-\lambda}|w-z|^\lambda\). Consequently,
\[
    f(w)^{\frac{1}{2}-|\Gamma|}\bigg|\prod_{\gamma\in\Gamma}\partial^\gamma f(w)-\prod_{\gamma\in\Gamma}\partial^\gamma f(z)\bigg|\leq Cr(y)^{\frac{k+\alpha}{2}-|\beta|-\lambda}|y-z|^\lambda=Cr(y)^{\Lambda-|\beta|}|y-z|^\lambda.
\]
Now from \eqref{eq:ctrl} and the definition of the pointwise semi-norm, we observe for any \(y\in 2Q\) that
\[
    [\partial^\beta\sqrt{f}]_\lambda(y)=\limsup_{w,z\rightarrow y}\frac{|\partial^\beta\sqrt{f(w)}-\partial^\beta\sqrt{f(z)}|}{|w-z|^\lambda}\leq Cr(y)^{\Lambda-|\beta|}
\]
Critically, the constant \(C\) is independent of \(Q\). Hence, properties \textit{(2)} and \textit{(3)} hold as claimed.
\end{proof}

If \(k\) is even and \(\beta\) is a multi-index of order \(\frac{k}{2}\), then the preceding result shows that \([\partial^\beta \sqrt{f}]_\frac{\alpha}{2}(y)\) is uniformly bounded on \(2Q\), meaning that \([\partial^\beta \sqrt{f}]_{\frac{\alpha}{2},2Q}\) is finite. A similar bound holds for odd-order derivatives and as a consequence we have the following.

\begin{cor}
Let \(f\in C^{k,\alpha}(\mathbb{R}^n)\) be non-negative and let \(r\) be as in \eqref{eq:controlfunc}. If \(\nu\) is a small positive constant and \(f(x)\geq \omega\nu r_Q^{k+\alpha}\), then \(\sqrt{f}\in C^\frac{k+\alpha}{2}(2Q)\).
\end{cor}

On the other hand, if \(f(x)<\omega\nu r_Q^{k+\alpha}\) then this corollary fails. The best we can do is form a local decomposition of \(f\) on \(2Q\) that takes the form \(f=g^2+F,\) where \(g\) is half as regular as \(f\) and the remainder \(F\in C^{k,\alpha}(2Q)\) depends on \(n-1\) variables. The construction of this local decomposition is patterned after the arguments in \cite{SOS_I,Tataru}, but we go to greater lengths to obtain estimates for every \(k\geq 0\), not just small values of \(k\).

In what follows, given \(x\in\mathbb{R}^n\) we write \(x=(x',x_n)\) for \(x'\in\mathbb{R}^{n-1}\) and \(x_n\in\mathbb{R}\). We also let \(Q'=\{x':x\in Q\}\) to denote the projection of \(Q\) onto \(\mathbb{R}^{n-1}\) in the \(x_n\) direction. Equipped with this notation we have the following result.

\begin{lem}\label{lem:minia}
Let \(f\in C^{k,\alpha}(\mathbb{R}^n)\) be non-negative and satisfy \(f(x)<\omega \nu r_Q^{k+\alpha}\). Assume also that \(r\) satisfies the following pointwise bound everywhere:
\vspace{-0.25em}
\begin{equation}\label{eq:needed}
    r(y)\leq \max\bigg\{f(y)^\frac{1}{k+\alpha},\sup_{|\xi|=1}[\partial^2_\xi f(y)]_+^\frac{1}{k-2+\alpha}\bigg\}.
\end{equation}
Then after a suitable change of variables, there exists a function \(X\in C^{k-1,\alpha}(2Q')\) which enjoys the following properties for \(y\in 2Q\),
\begin{itemize}
    \item[(1)] \(|x_n-X(y')|\leq r(y)\),
    \item[(2)] \(f(y',X(y'))\leq f(y)\),
    \item[(3)] \(\partial_{x_n}f(y',X(y'))=0\).
\end{itemize}
Moreover, for all derivatives of order \(|\beta|\leq k-1\) and \(y\in 2Q\), the function \(X\) satisfies 
\begin{itemize}
    \item[(4)] \(|\partial^\beta X(y')|\leq Cr(y)^{1-|\beta|}\),
    \item[(5)] \([\partial^\beta X]_{\alpha}(y')\leq Cr(y)^{1-\alpha-|\beta|}\).
\end{itemize}
\end{lem}

\textit{Remark}: This result states that \(f\) has a unique minimum along each ray in the \(x_n\) direction through a box contained in \(2Q\). Moreover, the collection of these minima, viewed as a function on \(2Q'\), belongs to \(C^{k-1,\alpha}(2Q')\). The bound \eqref{eq:needed} holds automatically when \(k=2\) and \(k=3\), while if \(k\geq 4\) additional conditions must be placed on \(f\) for \eqref{eq:needed} to hold.

\begin{proof}
Adapting an argument from \cite{SOS_I}, we let \(y'\in 2Q'\) and set \(h=\sqrt{6\omega\nu}\), where \(\nu\) is chosen small enough that \(h\leq 1\) and \(\omega\nu\leq 1\). First we show that the map \(t\mapsto f(y',t)\) has a unique minimum in the interval \((x_n-hr_Q,x_n+hr_Q)\), so that properties \textit{(1)-(3)} follow as direct consequences. Observe that if \(f(x)<\omega\nu r_Q^{k+\alpha}\leq r_Q^{k+\alpha}\) then \(r_Q>0\), and by \eqref{eq:needed} we have
\[
    r_Q=\sup_{|\xi|=1}[\partial^2_\xi f(x)]_+^\frac{1}{k-2+\alpha}.
\]
Assume without loss of generality that the supremum above is attained by the \(x_n\) directional derivative of \(f\) so that \(r_Q^{k-2+\alpha}=\partial^2_{x_n}f(x)\). We note that non-negativity of \(f\) and inclusion in \(C^{k,\alpha}(\mathbb{R}^n)\) are preserved by such rotations.

For each \((y',t)\in 2Q'\times [x_n-hr_Q,x_n+hr_Q]\) we have \(|(t,y')-x|\leq \sqrt{\nu^2+6\omega\nu}r_Q\), and so \(\partial^2_{x_n}f(y',t)\geq\frac{1}{2}r_Q^{k-2+\alpha}\) by item \textit{(2)} of \Cref{cor:supplementary}. It follows that \(t\mapsto f(x',t)\) has a unique minimum on \([x_n-hr_Q,x_n+hr_Q]\). Assume toward a contradiction that the minimum on this interval occurs at the left endpoint \(x_n-hr_Q\), so that \(\partial_{x_n}f(x',x_n-hr_Q)\geq 0\). Additionally, note that \(f(y',x_n)<\frac{3}{2}\omega\nu r_Q^{k+\alpha}\) due to property \textit{(1)} of \Cref{cor:supplementary} and our condition that \(f(x)<\omega\nu r_Q^{k+\alpha}\). So for some \(t\in (x_n-hr_Q,x_n)\) we have
\[
     \frac{3}{2}\omega\nu r_Q^{k+\alpha}>f(y',x_n)=f(y',x_n-hr_Q)+hr_Q\partial_{x_n}f(x',x_n-hr_Q)+\frac{h^2r_Q^2}{2}\partial^2_{x_n}f(y',t)\geq \frac{3}{2}\omega\nu r(x)^{k+\alpha},
\]
a contradiction since \(r_Q>0\). An identical contradiction arises if we assume that the minimum occurs at \(x_n+hr_Q\), so \(t\mapsto f(y',t)\) has a unique minimum in \((x_n-hr_Q,x_n+hr_Q)\) which we call \(X(y')\). Properties \textit{(1)-(3)} follow from our construction, and as \(y'\) was any point in \(2Q'\), we see that \(X\) is a well-defined function on \(2Q'\).

To see that \(X\in C^{k-1}(2Q')\), we observe that for each \(y'\) we have \(\partial_{x_n}f(y',X(y'))=0\) and that \(\partial^2_{x_n}f(y',X(y'))>0\) by property \textit{(2)} of \Cref{cor:supplementary}. Applying Theorem \ref{thm:IFT} to \(\partial_{x_n}f\) at the point \((y',X(y'))\) we see that \(X\in C^{k-1}(U)\) for some neighbourhood \(U\) of \(y'\), since \(\partial_{x_n}f\in C^{k-1}(\mathbb{R}^n)\). Further, as \(y'\) was any point in \(2Q'\) we can find such a neighbourhood around every point in \(2Q\). By the uniqueness statement of \Cref{thm:IFT}, these functions must agree with \(X\) an all such neighbourhoods, meaning that \(X\in C^{k-1}(2Q')\).

To establish derivative estimates of \(X\) for property \textit{(4)}, we first use the derivative formula in Theorem \ref{thm:IFT} to write
\[
    |\partial^\beta X(y')|=\bigg|\frac{1}{\partial^2_{x_n}f(y',X(y'))}\sum_{0\leq\eta\leq\beta}\sum_{ \substack{\Gamma\in P(\eta),\\\Gamma\neq \{\beta\}}}C_{\beta,\Gamma}(\partial^{\beta-\eta}\partial^{|\Gamma|+1}_{x_n}f(y',X(y')))\prod_{\gamma\in\Gamma}\partial^\gamma X(y')\bigg|.
\]
We simplify this by observing that \(\partial^2_{x_n}f(y',X(y'))\geq Cr(y)^{k-2+\alpha}\) on \(2Q'\) by property \textit{(2)} of \Cref{cor:supplementary} and slow variation of \(r\). Using this fact with the estimates of Theorem \ref{thm:cauchylike}, we get
\[
    |\partial^\beta X(y')|\leq Cr(y)^{2-k-\alpha}\sum_{0\leq\eta\leq\beta}\sum_{ \substack{\Gamma\in P(\eta),\\\Gamma\neq \{\beta\}}}r(y',X(y'))^{k+\alpha+|\eta|-|\beta|-|\Gamma|-1}\prod_{\gamma\in\Gamma}|\partial^\gamma X(y')|.
\]
Taking \(\nu\) small enough that \Cref{lem:slowvar} holds when \(\nu\) is replaced by \(\sqrt{\nu^2+6\omega\nu}\), it follows from the definition of \(X\) that \(r(y',X(y'))\leq Cr(y)\) whenever \(y'\in 2Q'\).

Now we argue by strong induction that \(|\partial^\beta X(y')|\leq Cr(y)^{1-|\beta|}\), observing for a base case that if \(|\beta|=1\) then for some \(j\in\{1,\dots,n-1\}\) and any \(y'\in 2Q'\) we can write 
\[
    |\partial^\beta X(y')|=|\partial_{x_j}X(y')|=\bigg|\frac{\partial_{x_j}\partial_{x_n}f(y',X(y'))}{\partial^2_{x_n}f(y',X(y'))}\bigg|\leq \frac{Cr(y)^{k+\alpha-2}}{r(y)^{k+\alpha-2}}=C.
\]
For an inductive hypothesis we assume that \(|\partial^\gamma X(y')|\leq Cr(y)^{1-|\gamma|}\) whenever \(|\gamma|<|\beta|\), so that
\[
    |\partial^\beta X(y')|\leq C\sum_{0\leq\eta\leq\beta}\sum_{ \substack{\Gamma\in P(\eta),\\\Gamma\neq \{\beta\}}}r(y)^{1-|\beta|+|\eta|-|\Gamma|}\prod_{\gamma\in\Gamma}r(y)^{1-|\gamma|}=Cr(y)^{1-|\beta|}.
\]
It follows that the claimed estimate holds for every \(\beta\) with \(|\beta|\leq k-1\), giving \textit{(4)}.

It remains to demonstrate that estimate \textit{(5)} holds, and once again we prove this using strong induction. For a base case we observe that \(|X(w')-X(z')|\leq C|w'-z'|\) by the Mean Value Theorem and the fact that \(|\nabla X|\leq C\) uniformly. Further, if \(w',z'\in 2Q'\) then we have that \(|w'-z'|\leq Cr(y)^{1-\alpha}|w'-z'|^\alpha\) and
\[
    |X(w')-X(z')|\leq Cr(y)^{1-\alpha}|w'-z'|^\alpha,
\]
from which it follows by taking a limit supremum as \(w',z'\rightarrow y'\) that \([X]_{\alpha}(y')\leq Cr(y)^{1-\alpha}\) for \(y\in 2Q'\). Thus the required semi-norm estimate holds in the base case.

For the inductive step first use \Cref{thm:IFT} to write the pointwise difference \(\partial^\beta X(w')-\partial^\beta X(z')\) for \(w',z'\in 2Q'\) in expanded form as follows:
\[
    \sum_{\eta\leq\beta}\sum_{ \substack{\Gamma\in P(\eta),\\\Gamma\neq \{\beta\}}}C_{\beta,\Gamma}\bigg(\frac{\partial^{\beta-\eta}\partial^{|\Gamma|+1}_{x_n}f(w',X(w'))}{\partial^2_{x_n}f(w',X(w'))}\prod_{\gamma\in\Gamma}\partial^\gamma X(w')-\frac{\partial^{\beta-\eta}\partial^{|\Gamma|+1}_{x_n}f(z',X(z'))}{\partial^2_{x_n}f(z',X(z'))}\prod_{\gamma\in\Gamma}\partial^\gamma X(z')\bigg).
\]
For brevity, we write \(\Tilde{w}=(w',X(w'))\) and note that \(|\tilde{w}-\Tilde{z}|\leq C|w'-z'|\) by the estimates above. Using the triangle inequality we can bound each of the terms in the sum above by the quantity
\[
    \frac{|\partial^{\beta-\eta}\partial^{|\Gamma|+1}_{x_n}f(\tilde{w})|}{\partial^2_{x_n}f(\tilde{w})}\bigg|\prod_{\gamma\in\Gamma}\partial^\gamma X(w)-\prod_{\gamma\in\Gamma}\partial^\gamma X(z)\bigg|+\bigg|\frac{\partial^{\beta-\eta}\partial^{|\Gamma|+1}_{x_n}f(\tilde{w})}{\partial^2_{x_n}f(\tilde{w})}-\frac{\partial^{\beta-\eta}\partial^{|\Gamma|+1}_{x_n}f(\Tilde{z})}{\partial^2_{x_n}f(\Tilde{z})}\bigg|\prod_{\gamma\in\Gamma}|\partial^\gamma X(w)|
\]
meaning that the required semi-norm estimate on \(\partial^\beta X\) will follow if we can bound each of the four factors above in an appropriate fashion.

To this end we first observe that by lower control of \(\partial^2_{x_n}f\) on \(2Q'\), together with slow variation of \(r\) and the derivative estimates of Theorem \ref{thm:cauchylike}, we have
\[
    \frac{|\partial^{\beta-\eta}\partial^{|\Gamma|+1}_{x_n}f(\tilde{w})|}{\partial^2_{x_n}f(\tilde{w})}\leq \frac{Cr(y)^{k+\alpha-|\beta|+|\eta|-|\Gamma|-1}}{r(y)^{k+\alpha-2}}=Cr(y)^{1-|\beta|+|\eta|-|\Gamma|}.
\]
Next, we employ the derivative bounds on \(X\) proved above for item \textit{(5)}. Notice that in the sum above, since \(\Gamma\in P(\eta)\) for \(\eta\) satisfying \(|\eta|\leq |\beta|\leq k\) and \(\Gamma\neq\{\beta\}\), for each \(\gamma\in \Gamma \) we must have that \(|\gamma|\leq k-1\). Therefore \(|\partial^\gamma X(w')|\leq Cr(y)^{1-|\gamma|}\) uniformly on \(2Q'\) by \textit{(4)}, and it follows that
\[
    \bigg|\prod_{\gamma\in\Gamma}\partial^\gamma X(w')\bigg|\leq C\prod_{\gamma\in\Gamma}r(y)^{1-|\gamma|}=Cr(y)^{|\Gamma|-|\eta|}.
\]
Using lower control of \(\partial^2_{x_n}f\) by \(r(y)^{k-2+\alpha}\) once again, we can also make the the following estimate:
\[
    \bigg|\frac{\partial^{\beta-\eta}\partial^{|\Gamma|+1}_{x_n}f(\tilde{w})}{\partial^2_{x_n}f(\tilde{w})}-\frac{\partial^{\beta-\eta}\partial^{|\Gamma|+1}_{x_n}f(\Tilde{z})}{\partial^2_{x_n}f(\Tilde{z})}\bigg|\leq \frac{C|\partial^2_{x_n}f(\Tilde{z})\partial^{\beta-\eta}\partial^{|\Gamma|+1}_{x_n}f(\tilde{w})-\partial^2_{x_n}f(\tilde{w})\partial^{\beta-\eta}\partial^{|\Gamma|+1}_{x_n}f(\Tilde{z})|}{r(y)^{2k+2\alpha-4}}.
\]
Employing the triangle inequality and our local control of derivatives of \(f\), we can bound the numerator on the right-hand side above by a constant multiple of the following expression,
\[
    r(y)^{k-2+\alpha}|\partial^{\beta-\eta}\partial^{|\Gamma|+1}_{x_n}f(\tilde{w})-\partial^{\beta-\eta}\partial^{|\Gamma|+1}_{x_n}f(\Tilde{z})|+r(y)^{k+\alpha-|\beta|+|\eta|-|\Gamma|-1}|\partial^2_{x_n}f(\Tilde{z})-\partial^2_{x_n}f(\tilde{w})|.
\]
If \(|\beta|-|\eta|+|\Gamma|+1<k\) then the mean value theorem, our pointwise derivative estimates on \(f\), and the estimate \(|\tilde{w}-\Tilde{z}|\leq Cr(y)\) for \(w',z'\in 2Q'\), all give
\[
    |\partial^{\beta-\eta}\partial^{|\Gamma|+1}_{x_n}f(\tilde{w})-\partial^{\beta-\eta}\partial^{|\Gamma|+1}_{x_n}f(\Tilde{z})|\leq Cr(y)^{k-|\beta|+|\eta|-|\Gamma|-1}|w'-z'|^\alpha.
\]
On the other hand, if \(|\beta|-|\eta|+|\Gamma|+1=k\) this estimate holds since \(\partial^{\beta-\eta}\partial^{|\Gamma|+1}_{x_n}f\in C^\alpha(\mathbb{R}^n)\) and 
\[
    |\partial^{\beta-\eta}\partial^{|\Gamma|}_{x_n}f(\tilde{w})-\partial^{\beta-
    \eta}\partial^{|\Gamma|}_{x_n}f(\Tilde{z})|\leq C|\tilde{w}-\Tilde{z}|^\alpha\leq C|w'-z'|^\alpha=Cr(y)^{k-|\beta|+|\eta|-|\Gamma|-1}|w'-z'|^\alpha.
\]
An identical argument shows that \(|\partial^2_{x_n}f(z',X(z'))-\partial^2_{x_n}f(w',X(w'))|\leq Cr(y)^{k-2}|w'-z'|^\alpha\) when \(k\geq 2\), and altogether we find now that
\[
    \bigg|\frac{\partial^{\beta-\eta}\partial^{|\Gamma|+1}_{x_n}f(\tilde{w})}{\partial^2_{x_n}f(\tilde{w})}-\frac{\partial^{\beta-\eta}\partial^{|\Gamma|+1}_{x_n}f(\Tilde{z})}{\partial^2_{x_n}f(\Tilde{z})}\bigg|\leq Cr(x)^{1-\alpha-|\beta|+|\eta|-|\Gamma|}|w'-z'|^\alpha.
\]

One more term estimate remains before we can invoke an inductive hypothesis on the H\"older semi-norms of \(X\) to simplify our bounds. Iterating the triangle inequality gives
\[
    \bigg|\prod_{\gamma\in\Gamma}\partial^\gamma X(w')-\prod_{\gamma\in\Gamma}\partial^\gamma X(z')\bigg|\leq \sum_{\gamma\in\Gamma}\bigg(\prod_{\mu\in\Gamma\setminus\{\gamma\}}\sup_{B'}|\partial^\mu X|\bigg)|\partial^\gamma X(w')-\partial^\gamma X(z')|,
\]
and using the estimates established above for property \textit{(2)}, we can bound this further to get
\[
    \bigg|\prod_{\gamma\in\Gamma}\partial^\gamma X(w')-\prod_{\gamma\in\Gamma}\partial^\gamma X(z')\bigg|\leq C\sum_{\gamma\in\Gamma}r(y)^{|\Gamma|-1-|\eta|+|\gamma|}|\partial^\gamma X(w')-\partial^\gamma X(z')|. 
\]
Consequently, we can bound each term in our earlier expansion of \(\partial^\beta X(w')-\partial^\beta X(z')\) to get
\[
    |\partial^\beta X(w')-\partial^\beta X(z')|\leq C\sum_{\eta\leq\beta}\sum_{ \substack{\Gamma\in P(\eta)\\\Gamma\neq \{\beta\}}}\bigg(\sum_{\gamma\in\Gamma}\frac{|\partial^\gamma X(w')-\partial^\gamma X(z')|}{r(y)^{|\beta|-|\gamma|}}+r(y)^{1-\alpha-|\beta|}|w'-z'|^\alpha\bigg).
\]

Dividing the expression above by \(|w'-z'|^\alpha\) and taking a limit supremum as \(w',z'\rightarrow y'\), it follows from the definition of the pointwise H\"older semi-norm that
\[
    [\partial^\beta X]_\alpha(y')\leq C\sum_{\eta\leq\beta}\sum_{ \substack{\Gamma\in P(\eta),\\\Gamma\neq \{\beta\}}}\sum_{\gamma\in\Gamma}r(y)^{-|\beta|+|\gamma|}[\partial^\gamma X]_\alpha(y')+Cr(y)^{1-\alpha-|\beta|}.
\]
Finally, assume for a strong inductive hypothesis that \([\partial^\gamma X]_{\alpha}(y')\leq Cr(y)^{1-\alpha-|\gamma|}\) for \(y'\in 2Q'\) whenever \(|\gamma|<|\beta|\). It follows immediately that
\[
    [\partial^\beta X]_\alpha(y')\leq C\sum_{\eta\leq\beta}\sum_{ \substack{\Gamma\in P(\eta),\\\Gamma\neq \{\beta\}}}\sum_{\gamma\in\Gamma}r(y)^{-|\beta|+|\gamma|}r(y)^{1-\alpha-|\gamma|}+Cr(y)^{1-\alpha-|\beta|}=Cr(y)^{1-\alpha-|\beta|}.
\]
By induction this holds for every \(\beta\) with \(|\beta|\leq k-1\), showing that \textit{(5)} holds on \(2Q'\).
\end{proof}

Henceforth, we assume that \(\nu\) has been chosen small enough that the \(\nu\)-dependant estimates of the preceding results hold. Moreover we assume that \eqref{eq:needed} holds, and we recall that this is automatic when \(k=2\) and \(k=3\). Our reward for undergoing the hard work of defining \(X\) on \(2Q'\), and proving its various properties, is that for \(y\in 2Q\) we may follow what is by now a standard technique (see e.g. \cite{Fefferman-Phong,SOS_I,Tataru}) of defining
\[
    F(y)=f(y',X(y')).
\]
It follows from the construction above that \(f-F\geq 0\), with equality at local minima of \(f\). Now we establish pointwise properties of the functions \(F\) and \(f-F\).

\begin{lem}\label{lem:nice}
Assume that the hypotheses of \Cref{lem:minia} hold and define \(F(y)=f(y',X(y'))\). Then for \(\beta\) with \(|\beta|<\frac{k+\alpha}{2}\) and every \(y\in 2Q\),
\begin{itemize}
    \item[(1)] \(|\partial^\beta\sqrt{f(y)-F(y)}|\leq Cr(y)^{\frac{k+\alpha}{2}-|\beta|}\),
    \item[(2)] \([\partial^\beta\sqrt{f-F}]_{\frac{\alpha}{2}}(y)\leq Cr(y)^{\frac{k}{2}-|\beta|}\) if \(k\) is even,
    \item[(3)] \([\partial^\beta\sqrt{f-F}]_{\frac{1+\alpha}{2}}(y)\leq Cr(y)^{\frac{k-1}{2}-|\beta|}\) if \(k\) is odd.
\end{itemize}
\end{lem}

\begin{proof}
By the fundamental theorem of calculus we have \(f(y)-F(y)=(y_n-X(y'))^2H(y)\), where
\[
    H(y)=\int_0^1 (1-t)\partial^2_{x_n}f(y',ty_n+(1-t)X(y'))dt.
\]
Arguing as in the proof of \Cref{lem:minia} we have \(\frac{1}{2}r_Q^{k-2+\alpha}\leq \partial^2_{x_n}f(y',ty_n+(1-t)X(y'))\) for every \(t\in[0,1]\), meaning that \(H(y)\geq Cr(y)^{k-2+\alpha}\). Additionally, we can bound the derivatives of \(H\) on \(2Q\). To this end we first observe that
\[
    \partial^\beta H(y)=\int_0^1 (1-t)\partial^\beta[\partial^2_{x_n}f(y',ty_n+(1-t)X(y'))]dt.
\]
To evaluate the derivative inside the integral, we use the shorthand \(L(y,t)=ty_n+(1-t)X(y')\) and write \(\beta'=(\beta_1,\dots,\beta_{n-1})\), so that \Cref{lem:genchain} gives
\[
    \partial^\beta[\partial^2_{x_n}f(y',L(y,t))]=t^{\beta_n}\sum_{\mu\leq\beta'}\sum_{\Gamma\in P(\mu)}C_{\beta',\Gamma}(1-t)^{|\Gamma|}\partial^{\beta'-\mu}\partial_{x_n}^{2+\beta_n+|\Gamma|}f(y',L(y,t))\prod_{\gamma\in\Gamma}\partial^\gamma X(y').
\]
It follows from the triangle inequality that
\[
    |\partial^\beta H(y)|\leq  C\sum_{\mu\leq\beta'}\sum_{\Gamma\in P(\mu)}\int_0^1t^{\beta_n}(1-t)^{|\Gamma|+1}|\partial^{\beta'-\mu}\partial_{x_n}^{2+\beta_n+|\Gamma|}f(y',L(y,t))|dt\prod_{\gamma\in\Gamma}|\partial^\gamma X(y')|.
\]
Recall from estimate \textit{(4)} of \Cref{lem:minia} that \(|\partial^\gamma X(y')|\leq Cr(y)^{1-|\gamma|}\). Additionally, Theorem \ref{thm:cauchylike} and \Cref{lem:slowvar} give \(|\partial^{\beta'-\mu}\partial_{x_n}^{2+\beta_n+|\Gamma|}f(y',L(y,t))|\leq Cr(y)^{k+\alpha-2-|\Gamma|-|\beta|+|\mu|}\), so
\[
    |\partial^\beta H(y)|\leq  C\sum_{\mu\leq\beta'}\sum_{\Gamma\in P(\mu)}r(y)^{k+\alpha-2-|\Gamma|-|\beta|+|\mu|}\prod_{\gamma\in\Gamma}r(y)^{1-|\gamma|}=Cr(y)^{k+\alpha-2-|\beta|}.
\]

Next, using these bounds we can estimate derivatives of \(\sqrt{H}\) on \(2Q\). To this end we employ the chain rule, the lower bound \(H(y)\geq Cr(y)^{k-2+\alpha}\), and the derivative bound above to estimate
\[
    |\partial^\beta \sqrt{H(y)}| =\bigg|\sum_{\Gamma\in P(\beta)}C_{\beta,\Gamma}H(y)^{\frac{1}{2}-|\Gamma|}\prod_{\gamma\in\Gamma}\partial^\gamma H(y)\bigg|\leq Cr(y)^{\frac{k+\alpha}{2}-1-|\beta|}.
\]
Further, estimates \textit{(1)} and \textit{(4)} of \Cref{lem:minia} show that \(|\partial^\beta(y_n-X(y'))|\leq Cr(y)^{1-|\beta|}\) on \(2Q\), respectively when \(|\beta|=0\) and \(|\beta|>0\). Using the product rule we can thus bound derivatives of \(\sqrt{f(y)-F(y)}=(y_n-X(y'))\sqrt{H(y)}\) by writing
\[
    \partial^\beta\sqrt{f(y)-F(y)}=\sum_{\gamma\leq\beta}\binom{\beta}{\gamma}\partial^{\beta-\gamma}(y_n-X(y'))\partial^\gamma\sqrt{H(y)}.
\]
Employing the triangle inequality and the derivative bounds computed above, we find that
\[
    |\partial^\beta\sqrt{f(y)-F(y)}|\leq C\sum_{\gamma\leq\beta}|\partial^{\beta-\gamma}(y_n-X(y'))||\partial^\gamma\sqrt{H(y)}|\leq Cr(y)^{\frac{k+\alpha}{2}-|\beta|}.
\]
Since \(y\in 2Q\) was arbitrary, estimate \textit{(1)} of \Cref{lem:nice} follows.

For estimates \textit{(2)} and \textit{(3)} we once again denote by \(\Lambda\) the integer part of \(\frac{k}{2}\), and set \(\lambda=\frac{k+\alpha}{2}-\Lambda\) so that the claimed estimates will follow if we can show that \([\partial^\beta\sqrt{f-F}]_{\lambda}(y)\leq Cr(y)^{\Lambda-|\beta|}\). To this end we observe that if \(w,z\in 2Q\), then the mean value theorem and estimate \textit{(1)} give
\[
    |\partial^\beta\sqrt{f(w)-F(w)}-\partial^\beta\sqrt{f(z)-F(z)}|\leq Cr(y)^{\frac{k+\alpha}{2}-|\beta|-1}|w-z|\leq Cr(y)^{\Lambda-|\beta|}|w-z|^\lambda.
\]
It follows by taking a limit supremum as \(w,z\rightarrow y\) that \([\partial^\beta\sqrt{f-F}]_{\lambda}(y)\leq Cr(y)^{\Lambda-|\beta|}\).
\end{proof}

There remains one more property of \(F\) to establish before we can proceed to extend the local results of this section to global ones. These are derivative and local semi-norm estimates, which follow from the properties of \(X\) established in \Cref{lem:minia}.

\begin{lem}\label{lem:implicitest}
Assume that the hypotheses of \Cref{lem:minia} hold, and define \(X\) and \(F\) as above. If \(F(y)=f(y',X(y'))\) then for every \(y\in 2Q\) the following estimates hold:
\begin{itemize}
    \item[(1)] \(|\partial^\beta F(y)|\leq Cr(y)^{k+\alpha-|\beta|}\),
    \item[(2)] \([\partial^\beta F]_{\alpha}(y)\leq Cr(y)^{k-|\beta|}\).
\end{itemize}
\end{lem}

\begin{proof}
Given a multi-index \(\beta\) of order \(|\beta|\leq k\), we compute \(\partial^\beta F\) by first writing \(\beta=\mu+\rho\), for a multi-index \(\rho\leq \beta\) with \(|\rho|=1\), and with \(\mu=\beta-\rho\). It follows from the definition of \(X\) that
\[
    \partial^\beta F(y')=\partial^\mu[\partial^\rho f(y',X(y'))+\partial_{x_n}f(y',X(y'))\partial^\rho X(y')]=\partial^\mu[\partial^\rho f(y',X(y'))].
\]
To bound the \(\mu\)-derivative on the right-hand side, we employ \Cref{lem:genchain} to get
\[
    \partial^\beta F(y')=\partial^\mu[\partial^\rho f(y',X(y'))]=\sum_{\eta\leq\mu }\sum_{\Gamma\in P(\eta)}C_{\mu,\Gamma}(\partial^{\beta-\eta}\partial^{|\Gamma|}_{x_n}f(y',X(y')))\prod_{\gamma\in\Gamma}\partial^\gamma X(y').
\]
Then, observe by the triangle inequality that
\[
    |\partial^\beta F(y')|\leq C\sum_{\eta\leq\mu }\sum_{\Gamma\in P(\eta)}|\partial^{\beta-\eta}\partial^{|\Gamma|}_{x_n}f(y',X(y'))|\prod_{\gamma\in\Gamma}|\partial^\gamma X(y')|.
\]
Since \(|\mu|=|\beta|-1\leq k-1\), for each \(\gamma\) in the product above, we have that \(|\partial^\gamma X(y')|\leq Cr(y)^{1-|\gamma|}\) on \(2Q'\) by estimate \textit{(4)} of \Cref{lem:minia}. Similarly, Theorem \ref{thm:cauchylike} and slow variation of \(r\) allow us to bound the derivatives of \(F\) above uniformly by powers of \(r(y)\) to get property \textit{(1)}:
\[
    |\partial^\beta F(y')|\leq C\sum_{\eta\leq\mu }\sum_{\Gamma\in P(\eta)}r(y)^{k+\alpha-|\beta|+|\eta|-|\Gamma|}\prod_{\gamma\in\Gamma}r(y)^{1-|\gamma|}=Cr(y)^{k+\alpha-|\beta|}.
\]

For estimate \textit{(2)}, we take \(w',z'\in 2Q'\) and write \(\tilde{w}=(w',X(w'))\). As above, for some multi-index \(\mu\leq\beta\) of order \(|\beta|-1\) we can expand \(\partial^\beta F\) as follows:
\[
    \partial^\beta F(w')-\partial^\beta F(z')= \sum_{0\leq\eta\leq\mu}\sum_{\Gamma\in P(\eta)}C_{\beta,\Gamma}\bigg(\partial^{\beta-\eta}\partial^{|\Gamma|}_{x_n}f(\tilde{w})\prod_{\gamma\in\Gamma}\partial^\gamma X(w)-\partial^{\beta-\eta}\partial^{|\Gamma|}_{x_n}f(\Tilde{z})\prod_{\gamma\in\Gamma}\partial^\gamma X(z)\bigg).
\]
From here, an argument identical to that used to prove estimate \textit{(5)} of \Cref{lem:minia} shows that \([\partial^\beta F]_{\alpha}(y')\leq Cr(y)^{k-|\beta|}\) for every derivative of order \(|\beta|\leq k\) and every \(y'\in 2Q'\), giving \textit{(2)}.
\end{proof}

Later we will require an appropriate extension by zero of \(F\) to all of \(\mathbb{R}^{n-1}\). Thankfully, such a function is fairly easy to construct. Recall that \(\lambda<\frac{3}{2}\) is a fixed parameter from Theorem \ref{thm:party}.

\begin{lem}\label{lem:crudeext}
There exists a non-negative function in \( C^{k,\alpha}(\mathbb{R}^{n-1})\) which agrees with \(F\) on \(\lambda Q'\).
\end{lem}

\begin{proof} Let \(\phi:\mathbb{R}\rightarrow\mathbb{R}\) be a smooth non-negative function supported in \([-1,1]\) such that \(\phi=1\) on \([-\frac{3}{4},\frac{3}{4}]\). Given \(y'\in 2Q'\) define
\[
    \psi(y')=\prod_{j=1}^{n-1}\phi\bigg(\frac{y_j-x_j}{\ell(Q)}\bigg),
\]
so that if \(y'\in\lambda Q' \) then for each \(j\in\{1,\dots,n-1\}\) we have \(|y_j-x_j|\leq \frac{\lambda}{2}\ell(Q)\leq \frac{3}{4}\ell(Q)\) (recall that \(\lambda<\frac{3}{2}\)) and \(\phi((y_j-x_j)/\ell(Q))=1\). It follows that \(\psi=1\) on \(\lambda Q'\). Additionally, if \(|x_j-y_j|>\ell(Q)\) for any \(j\in\{1,\dots,n-1\}\) then \(\phi((y_j-x_j)\ell(Q))=0\), meaning that \(\psi\) is supported in \(2Q'\).

It remains to show that \(\psi F\in C^{k,\alpha}(\mathbb{R}^{n-1})\). The argument from the proof of Theorem \ref{thm:party} shows that \(|\partial^\beta\psi(y')|\leq Cr(y')^{-|\beta|}\) and \([\partial^\beta\psi]_\alpha(y')\leq Cr(y')^{-|\beta|-\alpha}\) for \(y\in 2Q'\). Combining these estimates on \(\psi\) with the inequalities from \Cref{lem:implicitest} and employing the sub-product rule, we get 
\[
    [\partial^\beta(\psi F)]_{\alpha}(y')\leq\sum_{\gamma\leq\beta}\binom{\beta}{\gamma}([\partial^{\beta-\gamma}\psi]_{\alpha}(y')|\partial^\gamma F(y')|+[\partial^{\beta-\gamma}F]_{\alpha}(y')|\partial^\gamma \psi(y')|)\leq Cr(y')^{k-|\beta|}.
\]
It follows that if \(|\beta|=k\) then \([\partial^\beta(\psi F)]_{\alpha}(y')\leq C\) uniformly in \(2Q'\), and since this holds trivially outside of \(2Q'\) we conclude that \(\psi F\in C^{k,\alpha}(\mathbb{R}^{n-1})\) as claimed.
\end{proof}

This completes our discussion of localized results restricted to a single dyadic cube. Now we consider a the function \(\psi_j^2f\), where \(\psi_j\) is one of the functions given by Theorem \ref{thm:party}. Our primary result concerning this function is the following lemma; in what follows, given a cube \(Q_j\) we denote its center by \(x_j\) and we write \(r_j=\inf_{Q_j}r\).

\begin{lem}\label{lem:rootscors}
Let \(f\in C^{k,\alpha}(\mathbb{R}^n)\) be non-negative, and let \(\psi_j\) be one of the partition functions from Theorem \ref{thm:party} supported in \(\lambda Q_j\). If \(f(x_j)\geq \omega r_j^{k+\alpha}\) then the following estimates hold in \(\mathbb{R}^n\):
\begin{itemize}
    \item[(1)] \(|\partial^\beta[\psi_j\sqrt{f}](y)|\leq Cr(y)^{\frac{k+\alpha}{2}-|\beta|}\),
    \item[(2)] \([\partial^\beta(\psi_j\sqrt{f})]_{\frac{\alpha}{2}}(y)\leq Cr(y)^{\frac{k}{2}-|\beta|}\) if \(k\) is even,
    \item[(3)] \([\partial^\beta(\psi_j\sqrt{f})]_{\frac{1+\alpha}{2}}(y)\leq Cr(y)^{\frac{k-1}{2}-|\beta|}\) if \(k\) is odd.
\end{itemize}
If \(f(x_j)<\omega r_j^{k+\alpha}\) and the hypotheses of \Cref{lem:minia} hold, then for \(F\) as in \Cref{lem:nice} we have
\begin{itemize}
    \item[(4)] \(|\partial^\beta[\psi_j\sqrt{f-F}](y)|\leq Cr(y)^{\frac{k+\alpha}{2}-|\beta|}\),
    \item[(5)] \([\partial^\beta(\psi_j\sqrt{f-F})]_{\frac{\alpha}{2}}(y)\leq Cr(y)^{\frac{k}{2}-|\beta|}\) if \(k\) is even,
    \item[(6)] \([\partial^\beta(\psi_j\sqrt{f-F})]_{\frac{1+\alpha}{2}}(y)\leq Cr(y)^{\frac{k-1}{2}-|\beta|}\) if \(k\) is odd.
\end{itemize}
\end{lem}

The proofs of these estimates are straightforward; they rely only on the estimates from Theorem \ref{thm:party}, Lemma \ref{lem:local1} and Lemma \ref{lem:implicitest}, together with the product and sub-product rules introduced in Section 2. We omit them for brevity. The lemma above gives us the following result.

\begin{cor}\label{cor:theniceone}
The functions \(\psi_j\sqrt{f}\) and \(\psi_j\sqrt{f-F}\) belong to \(C^\frac{k+\alpha}{2}(\mathbb{R}^n)\).
\end{cor}

\section{Proof of Theorem \texorpdfstring{\ref{main 1}}{}}

Using induction on \(n\), we now show that if \(k=2\) or \(k=3\) then any non-negative \(f\in C^{k,\alpha}(\mathbb{R}^{n})\) is a SOS of at most \(s_n\) functions \(g_1,\dots,g_{s_n}\) in \eqref{halfreg}, for a number \(s_n\geq m(n,k,\alpha)\) which we estimate after our construction. Further, we show that for each \(j\in\{1,\dots,s_n\}\) the function \(g_j\) satisfies
\begin{flalign}
    &\qquad|\partial^\beta g_j(y)|\leq Cr_f(y)^{\frac{k+\alpha}{2}-|\beta|}\textrm{ for \(y\in\mathbb{R}^n\)},&&\label{induc1}\\
    &\qquad[\partial^\beta g_j]_{\frac{\alpha}{2}}(y)\leq Cr_f(y)^{\frac{k}{2}-|\beta|}\textrm{ if \(k\) is even},\label{induc2}\\
    &\qquad[\partial^\beta g_j]_{\frac{1+\alpha}{2}}(y)\leq Cr_f(y)^{\frac{k-1}{2}-|\beta|}\textrm{ if \(k\) is odd}.\label{induc3}
\end{flalign}
The function \(r_f\) is as in \eqref{eq:controlfunc}, and we now include the subscript to emphasize dependence on \(f\).

For the base case \(n=1\), we show that every non-negative \(f\in C^{k,\alpha}(\mathbb{R})\) is a finite sum of half-regular squares for \(k=2\) or \(k=3\) and \(0<\alpha\leq 1\). A stronger result appears in \cite{BonyFr}, but it omits the derivative and semi-norm estimates that we require for our inductive argument, so we furnish a proof with the required estimates. We emphasize that the one-dimensional result is not new, and our decomposition requires \(s_1=4\) functions. For our later estimates of \(s_n\), we use the optimal constant \(m(1,k,\alpha)=2\) established in \cite{BonyFr}.

Fixing a non-negative function \(f\in C^{k,\alpha}(\mathbb{R})\) and using \Cref{thm:party}, we construct a partition of unity using the control function \(r_f\) defined in \eqref{eq:controlfunc}, so that we may write \(f\) as follows:
\[
    f=\sum_{j=1}^\infty\psi_j^2f.
\]
Fixing \(j\in\mathbb{N}\), we observe that at the center \(x_j\) of \(Q_j\), either \(f(x_j)\geq \omega\nu r_j^{k+\alpha}\) or \(f(x_j)<\omega\nu r_j^{k+\alpha}\).

In the first case, \Cref{lem:rootscors} shows that \(\psi_j\sqrt{f}\) is a \(C^\frac{k+\alpha}{2}(\mathbb{R})\) function that satisfies \eqref{induc1} through \eqref{induc3}. On the other hand, if \(f(x_j)<\omega\nu r_j^{k\alpha}\) then \Cref{lem:minia} shows that there exists a unique local minimum of \(f\) at a point \(X_j\) near \(x_j\), and taking \(F_j=f(X_j)\), \Cref{lem:rootscors} shows that
\[
    \psi_j\sqrt{f-F_j}\in C^\frac{k+\alpha}{2}(\mathbb{R}).
\] 
Moreover, \Cref{lem:rootscors} shows that \(\psi_j\sqrt{f-F_j}\) has the required properties \eqref{induc1} through \eqref{induc3}. Thus \(f\) can be decomposed into a sum of two squares on \(2Q_j\),
\[
    \psi_j^2f=(\psi_j\sqrt{f-F_j})^2+(\psi_j\sqrt{F_j})^2.
\]

Next we must show that \(\psi_j\sqrt{F_j}\in C^\frac{k+\alpha}{2}(\mathbb{R})\) and that it satisfies the required pointwise estimates; this is easy, since \(F_j\) is constant. By slow variation of \(r_f\) we have \(F_j\leq Cr_j^{k+\alpha}\), and it follows by property \textit{(3)} of \Cref{thm:party} that if \(y\in 2 Q_j\) then the following estimate holds:
\[
    |\partial^\beta[\psi_j\sqrt{F_j}](y)|=\sqrt{F_j}|\partial^\beta\psi_j(y)|\leq Cr_j^{\frac{k+\alpha}{2}-|\beta|}\leq Cr(x)^{\frac{k+\alpha}{2}-|\beta|}.
\]
This bound trivially holds outside the support of \(\psi_j\), meaning it holds everywhere as required.

Next we prove that \(\psi_j\sqrt{F_j}\) satisfies the semi-norm estimates \eqref{induc2} and \eqref{induc3}. To deal with the cases in which \(k\) is odd and even simultaneously, we once again define \(\Lambda\) as the integer part of \(\frac{k}{2}\) and set \(\lambda=\frac{k+\alpha}{2}-\Lambda\), so that it suffices to prove that \([\partial^\beta(\psi_j\sqrt{F})]_\lambda(y)\leq Cr(y)^{\Lambda-|\beta|}\). This follows at once from property \textit{(4)} of \Cref{thm:party} together with slow variation:
\[
    [\partial^\beta(\psi_j\sqrt{F_j})]_\lambda(y)=\sqrt{F_j}[\partial^\beta\psi_j]_\lambda(y)\leq Cr_j^\frac{k+\alpha}{2}r_j^{-\lambda-|\beta|}=Cr_j^{\Lambda-|\beta|}\leq Cr(y)^{\Lambda-|\beta|}.
\]
Consequently, \(\psi_j\sqrt{F_j}\) has the properties \eqref{induc1} through \eqref{induc3} required for our induction.

Regardless of the behaviour of \(f\) at \(x_j\), we see now that we can write \(\psi_j^2f\) as a SOS of at most two functions which satisfy the pointwise estimates \eqref{induc1} through \eqref{induc3}. After relabelling, we can therefore write
\begin{equation}\label{eq:infsum}
    f=\sum_{j=1}^\infty h_j^2
\end{equation}
for functions \(h_j\in C^\frac{k+\alpha}{2}(\mathbb{R})\). By estimate \textit{(5)} of \Cref{thm:party}, we can identify \(s_1=4\) sub-collections of functions in the sum above which enjoy pairwise disjoint supports. Fixing any of these sub-collections and indexing its functions by \(S\subseteq\mathbb{N}\) we can define
\begin{equation}\label{eq:recombsum}
    g=\sum_{j\in S}h_j,
\end{equation}
so that \(g\in C^\frac{k+\alpha}{2}(\mathbb{R}^n)\) and it satisfies \eqref{induc1} through \eqref{induc3}, since the functions in the sum enjoy pairwise-disjoint supports and their H\"older norms are uniformly bounded. Calling the functions associated to each sub-collection \(g_1,\dots,g_4\), we can write \(f=g_1^2+g_2^2+g_3^2+g_4^2\), again since the functions in each sub-collection enjoy disjoint supports. This completes the base case.

Now we proceed with the inductive step, assuming for our inductive hypothesis that for every non-negative function \(h\in C^{k,\alpha}(\mathbb{R}^{n-1})\) satisfying \eqref{eq:needed}, there exist \(g_1,\dots,g_{s_{n-1}}\) satisfying estimates \eqref{induc1} through \eqref{induc3} for which \(h=g_1^2+\cdots+g_{s_{n-1}}^2\). Now we let \(f\in C^{k,\alpha}(\mathbb{R}^n)\) be non-negative, and using \Cref{thm:party} we form the partition of unity induced by \(r_f\) to write
\[
    f=\sum_{j=1}^\infty\psi_j^2f.
\]
An identical argument to that employed for the base case shows that either \(\psi_j\sqrt{f}\) satisfies the required derivative estimates and belongs to the half-regular H\"older space, or we can write 
\[
    \psi_j^2f=(\psi_j\sqrt{f-F_j})^2+\psi_j^2F_j.
\]
From \Cref{lem:rootscors}, the first function on the right-hand side above satisfies \eqref{induc1}, \eqref{induc2} and  \eqref{induc3}, and it belongs to the half-regular H\"older space.

The remainder function \(F_j\) is non-negative and by \Cref{lem:crudeext} it can be extended from the support of \(\psi_j\) to a function in \(C^{k,\alpha}(\mathbb{R}^{n-1})\) which agrees with \(F_j\) on this support set. Thus we can identify \(F_j\) with its extension to \(\mathbb{R}^{n-1}\), and by the inductive hypothesis we can write
\begin{equation}\label{eq:inductivedecomps}
    F_j=\sum_{\ell=1}^{s_{n-1}}g_\ell^2.
\end{equation}
for \(g_1,\dots,g_{s_{n-1}}\in C^\frac{k+\alpha}{2}(\mathbb{R}^{n-1}) \) which satisfy the estimates \(|\partial^\beta g_\ell(y)|\leq Cr_{F_j}(y)^{\frac{k+\alpha}{2}-|\beta|}\) and \([\partial^\beta g_j]_{\lambda}(y)\leq Cr_{F_j}(y)^{\Lambda-|\beta|}\), where as usual \(\Lambda=\lfloor\frac{k}{2}\rfloor\) and \(\lambda=\frac{k+\alpha}{2}-\Lambda\). Further, the estimates proved in \Cref{lem:implicitest} show for \(y\in\lambda Q_j\) that
\[
    r_{F_j}(y')=\max\{F_j(y')^\frac{1}{k+\alpha},\sup_{|\xi|=1}[\partial^2_\xi F_j(y')]_+^\frac{1}{k-2+\alpha}\}\leq Cr_f(y),
\]
meaning that on the support of \(\psi_j\) the estimates for \(g_1,\dots,g_{s_{n-1}}\) hold when \(r_F\) is replaced by \(r_f\). Using these refined estimates and the derivative estimates of Theorem \ref{thm:party} in the sub-product rule, we find for each \(\ell\in\{1,\dots,s_{n-1}\}\) that \(\psi_jg_\ell\) satisfies \eqref{induc1} and \eqref{induc2} or \eqref{induc3}, depending on the parity of \(k\). It follows that we can write \(\psi_j^2F_j\) as a sum of at most \(s_{n-1}\) half-regular squares.

In summary, for each \(j\in\mathbb{N}\) we can write \(\psi_j^2f\) as a sum of \(s_{n-1}+1\) squares in \(C^\frac{k+\alpha}{2}(\mathbb{R}^n)\). Combining the squares obtained for each \(j\) and relabelling, we see now that we may write
\[
    f=\sum_{j=1}^\infty h_j^2,
\]
for \(h_j\) satisfying \eqref{induc1}-\eqref{induc3} everywhere. By property \textit{(4)} of \Cref{thm:party}, we can partition the functions in the sum above into \((4^n-2^n)(s_{n-1}+1)\) sub-collections of functions which enjoy pairwise disjoint supports. Recombining and relabelling these functions exactly as we did in the one-dimensional setting, we see that we can set \(s_n=(4^n-2^n)(s_{n-1}+1)\) to write
\[
    f=\sum_{j=1}^{s_n}g_j^2
\]
for functions \(g_1,\dots,g_{s_n}\) which inherit the required differential inequalities. Finally, since \(f\) was any non-negative \(C^{k,\alpha}(\mathbb{R}^n)\) function satisfying \eqref{eq:needed}, this completes our inductive step. This shows that if \(k=2\) or \(k=3\) then every \(C^{k,\alpha}(\mathbb{R}^n)\) function is a sum of half-regular squares.

It remains to bound \(s_n\) to obtain an estimate for \(m(n,k,\alpha)\) when \(k=2\) and \(k=3\). The work in \cite{BonyFr} shows that \(s_1=m(1,k,\alpha)=2\) for \(k\geq 2\), and from estimate \textit{(5)} of Theorem \ref{thm:party} we see that \(s_2=9(s_1+1)=27\). For \(n\geq 3\) we have the recursion \(s_{n}=(4^n-2^n)(s_{n-1}+1)\), from which it follows upon iterating that
\[
    s_n=s_2\prod_{\ell=3}^n(4^\ell-2^\ell)+\sum_{j=0}^{n-3}\prod_{\ell=n-j}^n(4^\ell-2^\ell)
\]
when \(k\geq 3\). To estimate the first product we observe that if \(n\geq 3\) then
\[
    \prod_{\ell=3}^n(4^\ell-2^\ell)=2^{n^2+n-6}\prod_{\ell=3}^n\bigg(1-\frac{1}{2^\ell}\bigg)\leq 2^{n^2+n}\bigg(\frac{7}{512}\bigg),
\]
similarly for \(0\leq j\leq n-3\) the product in the second sum is bounded by \(2^{n+2nj-j^2}(2^{n-j}-1)\), and it follows that
\[
    \sum_{j=0}^{n-3}\prod_{\ell=n-j}^n(4^\ell-2^\ell)\leq 2^n\sum_{j=0}^{n-3}2^{2nj-j^2}(2^{n-j}-1)=2^{n^{2}+n}\sum_{j=3}^{n}\frac{2^{j}-1}{2^{j^{2}}}\leq 2^{n^{2}+n}\sum_{j=3}^{\infty}\frac{2^{j}-1}{2^{j^{2}}}.
\]
Combining these estimates and bounding the series, we find that if \(n\geq 3\) then \(s_n<2^{n^2+n-1.3844}\). This implies the claimed estimate for \(m(n,k,\alpha)\), completing the proof of Theorem \ref{main 1}. \hfill\qedsymbol

In passing, we note that if \(G_n\) is the infinite graph encountered in the proof of Theorem \ref{thm:party} in \(n\) dimensions, then we have \(m(n,k,\alpha)\leq t_n\) where \(t_1=2\) and \(t_{n}=\chi(G_n)(t_{n-1}+1)\) for \(n\geq 2\). So, better estimates of \(\chi(G_n)\) enable improved bounds on decomposition size. We leave such improvements for a future work.

\section{Proof of Theorem \texorpdfstring{\ref{main 3}}{}}

After rescaling if necessary, the differential inequalities \eqref{diffeqs} imply that \eqref{eq:needed} holds. It follows that on each cube \(Q_j\) given by Theorem \ref{thm:party}, either \(\psi_j\sqrt{ f}\) is half-regular, or \(F_j\in C^{k,\alpha}{(\mathbb{R})}\) can be defined as above. By the result of Bony \cite{BonyFr}, \(F_j\) is a sum of squares of at most two half-regular functions, meaning that even if we cannot conclude that \(\psi_j\sqrt{ f}\) is half-regular, we can instead write \(\psi_j^2f=(\psi_j\sqrt{f-F})^2+h_1^2+h_2^2,\) where each function on the right-hand side belongs to \(C^\frac{k+\alpha}{2}(\mathbb{R}^n)\). It follows that \(f\) is an infinite sum of half-regular squares, and recombining as in earlier cases we see that it takes at most \(\chi(G_2)(m(1,k,\alpha)+1)=27\) squares to represent \(f\).

\section{Connection to Polynomial SOS}\label{sec:poly}

It is shown by Hilbert \cite{Hilbert} that if \(n\geq 3\) and \(k\geq 6\), or if \(n\geq 4\) and \(k\geq 4\), then there exist non-negative homogeneous polynomials (also called forms) in \(n\) variables of degree \(k\) which are not sums of squares of polynomials. Such polynomials are in \(C^{k,\alpha}(\mathbb{R}^n)\) for any \(\alpha>0\), since their \(k^{\mathrm{th}}\) order derivatives are constant. Momentarily, we adapt an argument from \cite{BonyEn} to show that no such polynomial is a SOS of functions in \eqref{halfreg}. Thus, if \(n\) and \(k\) are both large then there exist non-negative \(C^{k,\alpha}(\mathbb{R}^n)\) functions which are not sums of half-regular squares. This fact is well-known; various examples appear in \cite{BonyEn,SOS_I}.

Another connection between non-decomposable polynomials and sums of squares does not appear to be widely understood. Denoting by \(F(n,k)\) the real vector space of forms with real coefficients in \(n\) variables of degree \(k\), it is shown by Choi, Lam \& Reznik \cite{CLR} that the Pythagoras number of \(F(n,k)\) is bounded below by a positive constant \(c(n,k)\). In other words, there exists \(P\in F(n,k)\) which is a SOS of polynomials, but \(P\) is a sum of no fewer than \(c(n,k)\) squares. For \(\alpha>0\) we have \(P\in \Sigma(n,k,\alpha)\), and as we will show momentarily \(P\) can be a sum of no fewer than \(c(n,k)\) half-regular squares. It follows at once that \(m(n,k,\alpha)\geq c(n,k)\) for any \(\alpha\in(0,1]\).

Our main bridge between the algebraic perspective of polynomial SOS on the one hand, and the analytic perspective of H\"older spaces on the other, is the following result. It is modelled after \cite[Lem. 1.3]{BonyEn} and \cite[Lem. 5.2]{SOS_I}, and adapts both to our more general situation.

\begin{lem}\label{newcoolexciting}
Let \(P\) be a non-negative homogeneous polynomial of even degree \(k\), assume that \(P\) is not a SOS of fewer than \(m\) polynomials, and let \(\Omega\subset\mathbb{R}^n\) be a bounded open set containing the origin. Given \(C>0\), there exists \(\delta>0\) such that
\[
    \qquad\sup_{\Omega}\bigg|\sum_{j=1}^{m}g_j^2-P\bigg|<\delta\quad\implies\quad\sum_{j=1}^{m}\|g_j\|_{C_b^\frac{k+\alpha}{2}(\overline{\Omega})}>C.
\]
In other words, if \(P\) is a SOS of \(m\) functions on \(\Omega\), then at least one does not belong to \(C_b^\frac{k+\alpha}{2}(\overline{\Omega})\).
\end{lem}

\begin{proof}
Fix \(C>0\) and assume to the contrary that for all \(\delta>0\), there exist \(g_{1},\dots,g_{m}\) such that 
\[
    \sum_{j=1}^m\|g_{j}\|_{C_b^\frac{k+\alpha}{2}(\overline{\Omega})}\leq C\qquad\mathrm{and}\qquad\sup_{\Omega}\bigg|\sum_{j=1}^mg_{j}^2-P\bigg|<\delta.
\]
For each \(\ell\in\mathbb{N}\) choose \(g_{1,\ell},\dots,g_{q,\ell}\) that satisfy the inequalities above with \(\delta=\ell^{-1}\). Fix \(j\) and consider the sequence \(\{g_{j,\ell}\}_{\ell}\) which is uniformly bounded since \(\|g_{j,\ell}\|_{C_b^\frac{k+\alpha}{2}(\Omega )}\leq C\) for each \(\ell\).

By \Cref{lem:cptmb} the embedding \(C_b^\frac{k+\alpha}{2}(\overline{\Omega})\hookrightarrow C_b^\frac{k}{2}(\overline{\Omega})\) is compact, so there exists a subsequence of \(\{g_{j,\ell}\}_{\ell}\) that converges uniformly to a \(\frac{k}{2}\)-times continuously differentiable function \(g_j\) on \(\overline{\Omega}\). Passing to this subsequence if necessary, we see that pointwise in \(\Omega\),
\[
    \bigg|\sum_{j=1}^mg_j^2-P\bigg|= \lim_{\ell\rightarrow\infty}\bigg|\sum_{j=1}^mg_{j,\ell}^2-P\bigg|\leq\lim_{\ell\rightarrow\infty}\frac{1}{\ell}=0.
\]
Thus in \(\Omega\) we can represent \(P\) is a sum of squares of \(m\) functions in \(C^\frac{k}{2}(\Omega)\).

Expanding on an argument of Bony et. al. \cite{BonyEn}, we show that this is impossible. Near zero we can write \(g_j=p_j+r_j\) for \(p_j\) a polynomial of degree at most \(\frac{k}{2}\) and \(r_j=o(|x|^\frac{k}{2})\) as \(x\rightarrow0\). Explicitly, from Taylor's Theorem we have
\[
    p_j(y)=\sum_{0\leq |\gamma|\leq \frac{k}{2}}\frac{\partial^{\gamma} g_j(0)}{\gamma!}y^\gamma\quad\mathrm{and}\quad r_j(y)=\sum_{|\gamma|=\frac{k}{2}}\frac{|\gamma|}{\gamma!}y^\gamma\int_0^1(1-t)^{|\gamma|-1}(\partial^{\gamma}g_j(ty)-\partial^{\gamma}g_j(0))dt.
\]
Taking \(Q=p_1^2+\cdots+p_m^2\), it follows from these formulas that we can write
\[
    P-Q=\sum_{j=1}^m2p_jr_j+\sum_{j=1}^mr_j^2=o(|x|^{\frac{k}{2}+d}),
\]
where \(d\) is the degree of the lowest degree monomial appearing in any \(p_j\). \textit{A priori} we know only that \(d\geq 0\), however we can refine an estimate for \(d\) as follows. Denoting by \(r\) the function on the right-hand side above, we have that \(P=Q+r\) and it follows from homogeneity of \(P\) that for any \(\lambda>0\) and \(x\neq0\),
\[
    \frac{\lambda^\frac{k}{2}P(x)}{|x|^\frac{k}{2}}=\frac{Q(\lambda x)}{\lambda^\frac{k}{2}|x|^\frac{k}{2}}+\frac{r(\lambda x)}{|\lambda x|^\frac{k}{2}}.
\]
Taking a limit as \(\lambda\rightarrow0\), the term on the left vanishes, and from the decay \(r=o(|x|^\frac{k}{2})\) we get
\[
    \frac{1}{|x|^\frac{k}{2}}\lim_{\lambda\rightarrow0}\frac{Q(\lambda x)}{\lambda^\frac{k}{2}}=0.
\]
This implies that each monomial term in \(Q\) has degree at least \(\frac{k}{2}\), and since \(Q\) is a sum of squares the lowest order terms of each \(p_j\) cannot cancel, meaning that each monomial term of \(p_j\) has degree at least \(d\geq \frac{k}{4}\). It follows from the form of our remainders above that \(P-Q=o(|x|^{\frac{k}{2}+\frac{k}{4}})\).

To iterate this argument, we now fix \(x\neq 0\) and \(\lambda>0\), and we use homogeneity of \(P\) to write
\[
    \frac{\lambda^\frac{k}{4}P(x)}{|x|^\frac{3k}{4}}=\frac{Q(\lambda x)}{\lambda^\frac{3k}{4}|x|^\frac{3k}{4}}+\frac{r(\lambda x)}{|\lambda x|^\frac{3k}{4}}.
\]
It follows from the fact that \(r=o(|x|^\frac{3k}{4})\) that by taking a limit in the identity above we get
\[
    \frac{1}{|x|^\frac{3k}{4}}\lim_{\lambda\rightarrow0}\frac{Q(\lambda x)}{\lambda^\frac{3k}{4}}=0,
\]
meaning that each monomial term in \(Q\) has degree at least \(\frac{3k}{4}\), and each monomial in each \(p_j\) has degree at least \(d\geq \frac{3k}{8}\). The remainder formula now shows that \(P-Q=o(|x|^\frac{7k}{8})\). Iterating in this fashion, we see by induction that \(P-Q=o(|x|^{k(1-2^{-j})})\) for each \(j\in\mathbb{N}\), meaning that \(P-Q=o(|x|^k)\) and \(P-Q=0\) identically in some neighbourhood of the origin. But this means that \(P\) is a sum of squares of \(m\) forms, a contradiction. 
\end{proof}

In case \(P\) is not a sum of squares of any number of polynomials, we have the following.

\begin{cor}\label{cor:notsos}
Let \(P\) be a non-negative homogeneous polynomial of even degree \(k\) which is not a SOS of polynomials. Then \(P\) is not a SOS of \(C^\frac{k+\alpha}{2}(\mathbb{R}^n)\) functions for \(\alpha>0\).
\end{cor}

More interesting for our purposes is the following more detailed consequence.

\begin{cor}\label{cor:notbigsos}
Let \(P\) be a non-negative polynomial of even degree \(k\) which is not a SOS of fewer than \(m\) polynomials. Then \(P\) is not a SOS of fewer than \(m\) functions in \(C^\frac{k+\alpha}{2}(\mathbb{R}^n)\) for \(\alpha>0\).
\end{cor}

In passing, we note that Corollary \ref{cor:notsos} together with the existence result from Hilbert \cite{Hilbert} imply the following interesting result that resembles \cite[Thm. 1.2 (a) \& (c)]{BonyEn}.
\begin{thm}\label{thm:existencecounter}
Fix \(\alpha>0\) and assume that \(n\geq 3\) and \(k\geq 6\), or that \(n\geq 4\) and \(k\geq 4\). Then there exist \(C^{k,\alpha}(\mathbb{R}^n)\) functions (polynomials) which are not sums of squares of functions in \(C^\frac{k+\alpha}{2}(\mathbb{R}^n)\).
\end{thm}

\textit{Remark:} Any polynomial is in \(C^\infty(\mathbb{R}^n)\), so this result implies that if \(n\geq 3\) then there exist non-negative functions in \(C^\infty(\mathbb{R}^n)\) which are not sums of squares in \(C^\infty(\mathbb{R}^n)\). However, if the degree of \(P\) is larger than \(k\) then \(P\not\in C^{k,\alpha}(\mathbb{R}^n)\) for \(\alpha\leq 1\), except when \(P\) has degree \(k+1\) in which case \(P\in C^{k,1}(\mathbb{R}^n)\). This means that \(C^\infty(\mathbb{R}^n)\not\subset C^{k,\alpha}(\mathbb{R}^n)\), motivating our use of \(C^{k,\alpha}(\mathbb{R}^n)\) in the theorem above.

\begin{proof}
If \(n\) and \(k\) satisfy the prescribed inequalities and \(k\) is even, then by \cite{Hilbert} there exists a non-negative homogeneous polynomial \(P\) of degree \(k\) which is not a sum of squares of polynomials. If \(\beta\) is a multi-index of order \(k\) then \(\partial^\beta P\) is constant, meaning that \([\partial^\beta P]_{\alpha,\mathbb{R}^n}=0\) and \(P\in C^{k,\alpha}(\mathbb{R}^n)\). By Corollary \ref{cor:notsos}, \(P\) is not a sum of squares of half-regular functions.

If \(k\) is odd, take \(P\) as above with even degree \(k-1\). If \(|\beta|=k\) then \(\partial^\beta P=0\), so once again \(P\in C^{k,\alpha}(\mathbb{R}^n)\). But for any \(\beta>0\) we recall that \(P\) is not a SOS of functions in the space \(C^{\frac{k-1}{2},\beta}(\mathbb{R}^n)\), and in particular \(P\) is not a SOS in \(C^{\frac{k-1}{2},\frac{1+\alpha}{2}}(\mathbb{R}^n)=C^\frac{k+\alpha}{2}(\mathbb{R}^n)\).
\end{proof}

\section{Proof of Theorem \texorpdfstring{\ref{main 2}}{}}

To begin, we quote a result from \cite{CLR}. Recalling that \(F(n,k)\) denotes the real vector space of degree \(k\) forms in \(n\) variables, we expanding the lower bound for \(P(n,k)\) given in \cite[Thm. 6.4]{CLR}, we obtain the following result.

\begin{thm}[Choi, Lam, Reznik]
If \(P(n,k)\) is the Pythagoras number of \(F(n,k)\), then 
\[
    P(n,k)\geq 2^{n-1}\prod_{j=1}^{n-1}\frac{j+k}{2j+k}\geq 2^{n-1}\left(\frac{1+k}{1+k+n}\right)^{\frac{n}{2}}
\]
\end{thm}

This result implies that for each even \(k\) and \(n\in\mathbb{N}\), there exists a homogeneous polynomial \(P\) of degree \(k\) in \(n\) variables which is a SOS of no fewer than \(a\) polynomials, where
\[
    a=2^{n-1}\prod_{j=1}^{n-1}\frac{j+k}{2j+k}.
\]
For any \(\alpha>0\) we have \(P\in C^{k,\alpha}(\mathbb{R}^n)\), and since \(P\) is a sum of squares of forms (which necessarily have degree \(\frac{k}{2}\)), we see also that \(P\) is a SOS in \(C^\frac{k+\alpha}{2}(\mathbb{R}^n)\) and \(P\in\Sigma(n,k,\alpha)\). On the other hand, Corollary \ref{cor:notbigsos} indicates \(P\) is a SOS of no fewer than \(m\) functions in \(C^\frac{k+\alpha}{2}(\mathbb{R}^n)\), meaning that \(\ell(P)\geq a\). It follows that
\[
    m(n,k,\alpha)=\sup\{\ell(f)\;:\;f\in \Sigma(n,k,\alpha)\}\geq \ell(P)\geq a,
\]
giving the desired bound for even \(k\). If \(k\) is odd, it suffices to repeat the argument with a polynomial of degree \(k-1\) and to replace \(k\) with \(k-1\) in the expression for \(a\) above.\hfill\qedsymbol

\section{Proof of Theorem \texorpdfstring{\ref{main 4}}{}}

Fix non-negative \(f\in C^{1,\alpha}(\mathbb{R}^n)\) and observe that if \(x,h\in\mathbb{R}^n\) then a Taylor expansion gives
\[
    0\leq f(x)+h\cdot\nabla f(x) +[\nabla f]_\alpha|h|^{1+\alpha}.
\]
Since \(f\) is non-negative, an identical estimate holds if we replace \(h\) with \(-h\), showing that \(|h\cdot\nabla f(x)|\leq f(x) +[\nabla f]_\alpha|h|^{1+\alpha}\). Taking \(h=\lambda\nabla f(x)\) for \(\lambda>0\) gives \(h\cdot\nabla f(x)=\lambda|\nabla f(x)|^2\) and
\[
    \lambda|\nabla f(x)|^2\leq f(x) +\lambda^{1+\alpha}[\nabla f]_\alpha|\nabla f(x)|^{1+\alpha}.
\]

Choosing \(\lambda=(f(x)/\alpha[\nabla f]_\alpha)^\frac{1}{1+\alpha}/|\nabla f(x)|\) gives \(|\nabla f(x)|\leq [\nabla f]_\alpha^\frac{1}{1+\alpha}f(x)^\frac{\alpha}{1+\alpha}(\alpha^\frac{1}{1+\alpha}+\alpha^{-\frac{\alpha}{1+\alpha}})\), from which the claimed inclusion follows by the mean value theorem. \hfill\qedsymbol

\bibliographystyle{siam}
\pagestyle{plain}
\bibliography{references}
\end{document}